\documentclass[11pt]{article}
\usepackage{amsmath,amsthm,amssymb}
\usepackage{times}
\usepackage{enumerate}
\usepackage{epstopdf}
\usepackage{graphicx}
\usepackage[all,cmtip]{xy}

\def\titlerunning#1{\gdef\titrun{#1}}
\makeatletter
\def\author#1{\gdef\autrun{\def\and{\unskip, }#1}\gdef\@author{#1}}
\def\address#1{{\def\and{\\\hspace*{18pt}}\renewcommand{\thefootnote}{}%
\footnote {#1}}%
\markboth{\autrun}{\titrun}}
\makeatother
\def\email#1{e-mail: #1}

\def\keywords#1{\par\medskip
\noindent\textbf{Keywords.} #1}

\newtheorem{thm}{Theorem}[section]
\newtheorem{cor}[thm]{Corollary}
\newtheorem{lem}[thm]{Lemma}
\newtheorem{prop}[thm]{Proposition}

\theoremstyle{definition}
\newtheorem{defin}[thm]{Definition}
\newtheorem{rem}[thm]{Remark}

\numberwithin{equation}{section}

\newcommand{\R}{\mathbb{R}}
\newcommand{\C}{\mathbb{C}}
\newcommand{\T}{\mathbb{T}}
\newcommand{\Z}{\mathbb{Z}}
\newcommand{\N}{\mathbb{N}}
\newcommand{\cc}{\mathbb{\mathcal C}}
\newcommand{\gfqi}{$gfqi$ }
\newcommand{\ham}{\textnormal{Ham}}
\newcommand{\pd}{\textnormal{\small PD}}
\newcommand{\nbd}{neighbourhood }
\newcommand{\im}{\textnormal{Im}\,}

\newcommand{\cqfd}{\hfill $\square$ \vspace{0.1cm}\\ }
\newcommand{\id}{\textnormal{Id}\,}
\newcommand{\priv}{\backslash}
\newcommand{\grad}{\vec{\nabla}}
\newcommand{\supp}{\text{Supp}\,}

\DeclareMathOperator{\area}{area}
\DeclareMathOperator{\Vol}{Vol}
\DeclareMathOperator{\Red}{Red}
\DeclareMathOperator{\Ham}{Ham}

\begin{document}




\titlerunning{Coisotropic camel}

\title{Middle dimensional symplectic rigidity and its effect on Hamiltonian PDEs}

\author{Jaime Bustillo}

\date{}

\maketitle

\address{J. Bustillo: D\'epartement de Math\'ematiques et Applications,
	\'Ecole Normale Sup\'erieure, CNRS, PSL Research University,
	45 rue d'Ulm, 75005 Paris, France; \email{bustillo@dma.ens.fr}}


\begin{abstract}
In the first part of the article we study Hamiltonian diffeomorphisms of $\R^{2n}$ which are generated by sub-quadratic Hamiltonians and prove a middle dimensional rigidity result for the image of coisotropic cylinders. The tools that we use are Viterbo's symplectic capacities and a series of inequalities coming from their relation with symplectic reduction. In the second part we consider the nonlinear string equation and treat it as an infinite-dimensional Hamiltonian system. In this context we are able to apply Kuksin's approximation by finite dimensional Hamiltonian flows and prove a PDE version of the rigidity result for coisotropic cylinders. As a particular example, this result can be applied to the Sine-Gordon equation.

\end{abstract}

\section{Introduction}
Consider $\C^n=\R^{2n}$ with coordinates given by $(q_1,p_1,\dots,q_n,p_n)$ and let $\omega=\sum_i dq_i\wedge dp_i$ be the standard symplectic form. Gromov's non-squeezing theorem \cite{gromov1985pseudo} states that symplectic diffeomorphisms of $(\C^{n},\omega)$ cannot send balls of radius $r$ into symplectic cylinders of radius $R$ unless $r\leq R$. For example if $\phi$ is a symplectic diffeomorphism, $B^{2n}_r$ is the ball of radius $r$ in $\C^n$ and $B^2_R\subseteq \C$ is the two dimensional disc of radius $R$ then $$\phi(B^{2n}_r)\subseteq B^2_R\times \C^{n-1}\quad\text{implies}\quad r\leq R.$$ The original proof relied on the technique of pseudo-holomorphic curves which was later used to prove a wide range of results in symplectic geometry. Shortly after, several authors \cite{ekeland1990symplectic,hofer1990new,viterbo1992symplectic} gave independent proofs of Gromov's theorem using the concept of symplectic capacities. A symplectic capacity is a function $c:\mathcal{P}(\C^n)\rightarrow [0,+\infty]$ that verifies the following properties:\begin{enumerate}
	\item(monotonicity) If $U\subseteq V$ then $c(U)\leq c(V)$.
	\item(conformality) $c(\lambda U)=\lambda^2c(U)$ for all $\lambda\in \R$.
	\item(symplectic invariance) If $\phi:\C^n\rightarrow \C^n$ is a symplectic diffeomorphism then $c(\phi(U))=c(U)$.
	\item(non-triviality+normalization) $c(B^{2n}_1)=\pi=c(B^2_1\times \C^{n-1})$.
\end{enumerate}
Together, the existence of a function with these properties implies Gromov's theorem. In this article we are going to work with Viterbo's capacities in order to prove a rigidity result for a particular type of Hamiltonian diffeomorphisms. More precisely we are interested in the middle dimensional symplectic rigidity problem. 

One of the first questions regarding this problem appeared in \cite{hofer1990capacities} where Hofer asked about the generalization of capacities to middle dimensions. He asked if there exists a k-intermediate symplectic capacity $c^k$ satisfying monotonicity, k-conformality, symplectic invariance and $$c^k(B^{2k}_1\times \C^{n-k})<+\infty\quad\text{ but }\quad c^k(B^{2k-2}_1\times\C\times\C^{n-k})=+\infty?$$ In \cite{guth2008symplectic} Guth gave a partial answer to this question. He studied embeddings of polydisks an proved that $k$-capacities that verify the the following continuity property: $$\lim_{R\rightarrow \infty} c^k(B^{2k}_1\times B_R^{2n-2k})<+\infty\quad\text{ but }\quad \lim_{R\rightarrow \infty}c^k(B^{2k-2}_1\times B_R^{2n-2k+2})=+\infty?$$ do not exist. The question of less regular capacities was recently answered in the negative by Pelayo and V\~u Ng\d{o}c in \cite{pelayo2015hofer} using in part the ideas in \cite{guth2008symplectic}. In their article they proved that if $n\geq 2$ then $B^2_1\times \C^{n-1}$ can be symplectically embedded into the product $B^{2n-2}_R\times \C$ for $R=\sqrt{2^{n-1}+2^{n-2}-2}$. In particular, by monotonicity and homogeneity, the value of the capacity on the left has to be greater than or equal to a constant times the value on the right. 

Another point of view for the middle dimensional problem comes from a reformulation of Gromov's non-squeezing theorem. In dimension $2$ symplectomorphims are the same as area preserving maps so in \cite{eliashberg1991convex} Eliashberg and Gromov  pointed out that (using a theorem of Moser about the existence of area preserving diffeomorphisms) Gromov's theorem is equivalent to $$\area(\Pi_1\phi(B^{2n}_r))\geq \pi r^{2}\quad\text{for every symplectomorphism $\phi$}.$$ Denote by $\Pi_k$ the projection on the first $2k$ coordinates. A possible generalization of this statement to higher dimensions is $$\Vol(\Pi_k\phi(B^{2n}_r))\geq  \Vol(\Pi_kB^{2n}_r)= \Vol(B^{2k}_r)\quad\text{for every symplectomorphism $\phi$}.$$ This problem was studied by Abbondandolo and Matveyev in \cite{abbondandolo2013shadow}. In their article they proved that the volume with respect to $\omega^{\wedge k}$ of $\Pi_k\phi(B^{2n}_r)$ can be made arbitrarily small using symplectomorphisms. This ruled out the existence of middle dimensional volume symplectic rigidity for the ball. Nevertheless they proved that the rigidity exists in the linear case and, shortly after, several local results appeared: In \cite{rigolli2015middle} Rigolli proved that there is local middle dimensional volume rigidity if one restricts the class of symplectomorphisms to analytic ones, and in \cite{abbondandolo2018sharp} Abbondandolo, Bramham, Hryniewicz and Salom\~ao proved that the same kind of local rigidity appears for smooth symplectomorphisms in the case $k=2$.

We would like to point out another possible middle dimensional generalization of the squeezing problem. In dimension $2$ the value of any symplectic capacity on topological discs coincides with the area, so one may also rewrite Gromov's theorem as $$c(\Pi_1\phi(B^{2n}_r))\geq \pi r^{2}\quad\text{for every symplectomorphism $\phi$},$$ where $c$ is a symplectic capacity. One can then ask if this inequality is true with $\Pi_1$ replaced by $\Pi_k$, and more generally look at subsets $Z$ different from $B^{2n}_r$ and replace $\pi r^2$ with the capacity of $\Pi_kZ$. We prove that this type of inequality is true for $Z= X\times\mathbb{R}^{n-k}\subseteq \mathbb{C}^k\times \mathbb{C}^{n-k}$ provided that we restrict the class of symplectomorphisms. The maps the we consider are Hamiltonian diffeomorphisms $\psi=\psi_1^H$ generated by sub-quadratic Hamiltonians $H$, that is, by functions $H$ that verify $$\lim_{\lvert z\rvert\rightarrow +\infty}\frac{\lvert\nabla H_t(z)\rvert}{\lvert z\rvert}=0 \quad\text{uniformly in }t.$$ More precisely, if we denote by $c$ and $\gamma$ the two symplectic capacities defined by Viterbo in \cite{viterbo1992symplectic}, we have the following theorem:

\begin{thm}[Coisotropic camel theorem]\label{thm:coisotropic nonsqueezing}
	Let $X\subset\C^k$ be a compact set. Consider $X\times\mathbb{R}^{n-k}\subseteq \mathbb{C}^k\times \mathbb{C}^{n-k}$ and let $\psi=\psi^H_1$ be a Hamiltonian diffeomorphism of $\C^n$ generated by a sub-quadratic Hamiltonian $H$. Then $$c(X)\leq \gamma(\Pi_k(\psi(X\times\R^{n-k})\cap \C^k\times i\R^{n-k})).$$
\end{thm}

Using the monotonicity of the capacity $\gamma$ we get as an immediate consequence that $$c(X)\leq \gamma(\Pi_k\psi(X\times\R^{n-k})).$$ For example if we take $X=B^{2k}_r$ we get the inequality $$\gamma(\Pi_k\psi(B_r^{2k}\times\R^{n-k}))\geq \pi r^2.$$ 

We want to point out that the subset on the right hand side of the inequality in Theorem \ref{thm:coisotropic nonsqueezing} is the symplectic reduction of $\psi(X\times\R^{n-k})$ by the transverse coisotropic subspace $\C^k\times i\R^{n-k}$. More precisely, recall that by definition a coistropic subspace $W\subseteq \C^n$ verifies $W^{\omega}\subseteq W$ where $W^{\omega}$ stands for symplectic orthogonal. One can then consider the space $W/W^{\omega}$ which is symplectic by construction. Denote by $\pi_W: W\rightarrow W/W^{\omega}$ the quotient map. The symplectic reduction of a subset $Z\subseteq \C^n$ by $W$ is defined as $\Red_W(Z)=\pi_W(Z\cap W)$. If we set $W=\C^k\times i\R^{n-k}$ then the inequality in Theorem \ref{thm:coisotropic nonsqueezing} can be rewritten as:
$$c(X)\leq\gamma(\Red_W(\psi(X\times \R^{n-k}))).$$
Note that this reduction is the projection of a bounded set so in particular Theorem \ref{thm:coisotropic nonsqueezing} is not trivial for compactly supported Hamiltonians (see Figure \ref{fig:lagrangian}).

\begin{figure}[h]
	
	\centering
	\includegraphics[width=0.9\textwidth]{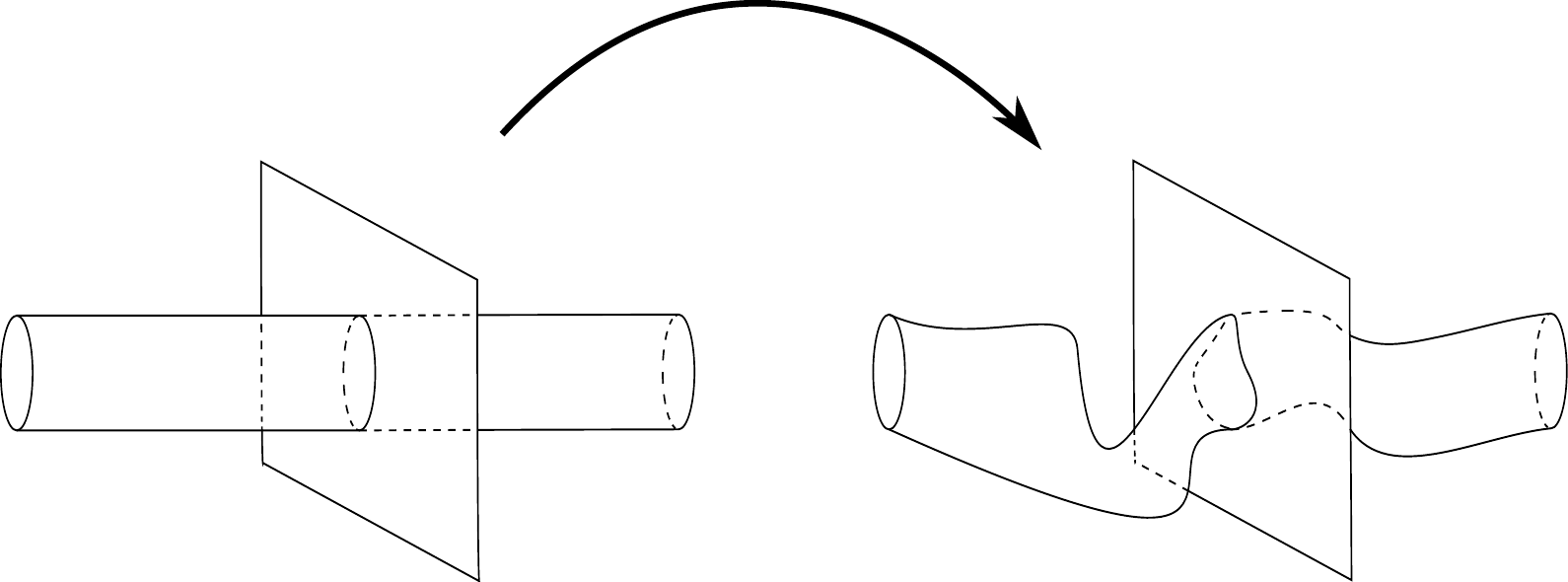}
	\label{fig:lagrangian}
	\caption{This figure represents the image of the coisotropic cylinder $X\times\mathbb{R}^{n-k}$ by a compactly supported Hamiltonian diffeomorphism $\psi$. The transverse plane represents the complementary coisotropic subspace $W=\C^k\times i\R^{n-k}$. Theorem \ref{thm:coisotropic nonsqueezing} gives information about the capacity of the projection of the intersection with $W$.}
\end{figure}

\begin{rem}
	Theorem \ref{thm:coisotropic nonsqueezing} is related to the classical camel theorem which states that there is no symplectic isotopy $\psi_t$ with support in $(\C^{n}\setminus(\C^{n-1}\times \R))\cup B^{2n}_\epsilon$, such that $\psi_0=\id$ and $\psi_1$ sends a ball of radius $r>\epsilon$ contained in one component of $\C^{n}\setminus(\C^{n-1}\times \R)$ to the other. In \cite{viterbo1992symplectic} Viterbo proved that if $\psi_t$ sends a ball from one connected component of $\C^{n}\setminus(\C^{n-1}\times \R)$ to the other then for $V=\bigcup_{t\in[0,1]}\psi_t(B^{2n}_r)$ and $W=\C^{n-1}\times \R$ we have $$\gamma(\Red_W(V))\geq \pi r^2.$$
\end{rem}

As an intuition of what these capacities are, if $K$ is a convex smooth body then $c(K)$ coincides with the minimal area of a closed caracteristic on $\partial K$. On the other hand $\gamma$ is defined using Viterbo's distance on $\Ham^c(\C^n)$; the energy of a diffeomorphism is then defined to be the distance to the identity and $\gamma(U)$ measures the minimal energy that one needs to displace $U$ from itself. Both capacities are always related by the inequality $c(X)\leq \gamma(X)$ which is recovered by Theorem \ref{thm:coisotropic nonsqueezing} when $\psi=\id$.

For general sets the construction of Viterbo's capacities (cf. \cite{viterbo1992symplectic}) starts by defining for the time-$1$ map $\psi=\psi^H_1$ of the flow of a compactly supported Hamiltonian $H_t$ two values: $c(1,\psi)$ and $c(\mu,\psi)$. These values correspond to the action value of certain $1$-periodic orbits of the flow obtained by variational methods. The bi-invariant metric on $\Ham^c(\C^n)$ is then defined as $d(\psi,\id)=\gamma(\psi)=c(\mu,\psi)-c(1,\psi)$. All these quantities are invariant by symplectic conjugation so they can be used to define symplectic invariants on open bounded sets as:
$$ c(U)=\sup\{c(\mu,\psi),\,\psi\in\text{Ham}^c(U)\}$$ 
$$\gamma(U)=\inf\{\gamma(\psi),\; \psi\in \ham^c(\C^{n}),\; \psi(U)\cap U=\emptyset\}$$
If $V$ is an open (not necessarily bounded) subset of $\C^n$ then $c(V)$ (resp. $\gamma(V)$) is defined as the supremum of the values of $c(U)$ (resp. $\gamma(U)$) for all open bounded $U$ contained in $V$. If $X$ is an arbitrary subset of $\C^n$ then its capacity $c(X)$ (resp. $\gamma(X)$) is defined as the infimum of all the values $c(V)$ (resp. $\gamma(V)$) for all open $V$ containing $X$. 

If $k>0$ one can prove that $c(\C^k\times \R^{n-k})=0=\gamma(\C^k\times \R^{n-k})$ for the coisotropic subspace $\C^{2k}\times \R^{n-k}\subseteq \C^k\times \C^{n-k}$ (see Appendix \ref{appendix: calculations}). By monotonicity the same is true for coisotropic cylinders $X\times \R^{n-k}\subseteq \C^k\times \C^{n-k}$ so the existence of Viterbo's capacities all alone does not provide rigidity information for the image of these sets by general symplectic diffeomorphisms.

The proof of Theorem \ref{thm:coisotropic nonsqueezing} is achieved by a series of inequalities between Viterbo's capacities of sets and the symplectic reduction of these sets. The advantage of using Viterbo's capacities is that they are constructed using generating functions and symplectic reduction can be seen as an explicit operation on generating functions which can be then studied in detail. We first prove the theorem for compactly supported Hamiltonian diffeomorphisms  and then reduce the general case to the compactly supported one. 

There is an unpublished proof of this theorem by Buhovski and Opshtein for the case $X=(S^1(r))^k$ a product of circles of radius $r$ and $\Red_W(\psi(X\times \R^{n-k}))\subseteq Z(R)$ a symplectic cylinder of radius $R$ which relies on the theory of pseudoholomorphic curves.

\begin{rem}\label{rem: squeezing} Theorem \ref{thm:coisotropic nonsqueezing} is \textit{not} true for general symplectomorphisms and its limits are well understood. An example of a symplectomorphism $\varphi$ which is not generated by a sub-quadratic Hamiltonian and which does not verify the weaker middle dimensional inequality $$c(X)\leq \gamma(\Pi_k\varphi(X\times\R^{n-k})),$$ is the symplectomorphism  $\varphi(z_1,\dots,z_n)=(z_{k+1},\dots, z_n,z_1,\dots,z_k).$ Indeed for this map $\varphi$ we have $$\varphi( X\times\mathbb{R}^{n-k})=\mathbb{R}^{n-k}\times X$$ and either the $k$-projection is contained in $Z=\R^k$ if $k\leq n/2$ or in $Z=\R^{n-k}\times \C^{4k-2n}$ otherwise. In both cases $\gamma(Z)=0$ by Appendix \ref{appendix: calculations}, so if for example $X$ is a closed ball, then the statement is not verified. In this example $\varphi$ is generated by a quadratic Hamiltonian. Moreover, we prove in Proposition \ref{prop: less than linear} that if $\lvert\nabla H_t(z)\rvert\leq A+ B\lvert z\rvert$ then its flow $\psi^H_t$ verifies the statement of Theorem \ref{thm:coisotropic nonsqueezing} at least for small times.\end{rem}

One may use the rigidity result of Theorem \ref{thm:coisotropic nonsqueezing} to define a non-trivial invariant. Consider the following quantity:
$$\gamma^k_G(U)=\inf\{\gamma(\Pi_k\phi(U)) \,|\,\phi\in G\}$$ where $G$ is a subgroup of the group of symplectic diffeomorphisms. For $G=Sympl(\C^n)$ we know by Remark \ref{rem: squeezing} that $\gamma^k_G$ is zero on coisotropic cylinders of dimension $k$. On the other hand, if the elements of $G$ are Hamiltonian diffeomorphisms generated by sub-quadratic functions then Theorem \ref{thm:coisotropic nonsqueezing} implies that $\gamma_G^k$ is bounded from below on coisotropic cylinders of dimension $k$ by the capacity $c$ of the base. As an example of $G$ one can take the subgroup of Hamiltonian diffeomorphisms $\varphi_t^H$ where $H$, $\varphi_t^H$ and $(\varphi_t^H)^{-1}$ are Lipschitz on the space variable over compact time intervals (see Appendix \ref{appendix: subgroup}). For this subgroup Theorem \ref{thm:coisotropic nonsqueezing} gives $$c(X)\leq \gamma^k_G(X\times\R^{n-k})\leq \gamma(X).$$

\paragraph{Hamiltonian PDEs.} The second part of this article deals with middle dimensional symplectic rigidity in infinite dimensional Hilbert spaces. 

Let $E$ be a real Hilbert space and let $\omega$ be a non-degenerate 2-form. Here we will understand non-degenerate in the sense that the map $\xi\in E\rightarrow \omega(\xi,\cdot)\in E^*$ is an isomorphism. In contrast with the finite dimensional case, little is known about the rigidity properties of symplectomorphisms in this context. The most general attempt to prove a non-squeezing theorem has been \cite{abbondandolo2015non} where the result is proved only for convex images of the ball. The first result pointing in the direction of the infinite dimensional equivalent of Gromov's theorem dates back to \cite{kuksin1995infinite}. Kuksin gave a proof of the theorem for a particular type of symplectomorphism that appear in the context of Hamiltonian PDEs. He did this by approximating the flows by finite dimensional maps and then applying Gromov's theorem. Since then there has been a great number of articles proving the same result for different Hamiltonian PDEs via finite dimensional approximation. We refer the reader to \cite{killip2016finite} for an excellent summary of the prior work.

The goal of the second part of this article is to extend Theorem \ref{thm:coisotropic nonsqueezing} to the infinite dimensional case. We restrict ourselves to semilinear PDEs of the type described in \cite{kuksin1995infinite}. Let $\langle\cdot,\cdot\rangle$ be the scalar product of $E$, $\{\varphi_j^{\pm}\}$ be a Hilbert basis, $J:E\rightarrow E$ be the complex structure defined by $J\varphi_j^\pm=\mp\varphi_j^\mp$ and $\bar J=-J$. The symplectic structure that we consider is  $\omega(\cdot,\cdot)=\langle \bar J\cdot,\cdot\rangle$ and the Hamiltonian functions are of the form $$H_t(u)=\frac{1}{2}\langle Au,u\rangle+h_t(u),$$ where $A$ is a (possibly unbounded) linear operator and $h_t$ is a smooth function. The Hamiltonian vector field is $$X_H(u)=JAu+J\nabla h_t(u)$$ Remark that the domain of definition of the vector field is the same as the domain of $A$ which is usually only defined on a dense subspace of $E$. If $e^{tJA}$ is bounded, then solutions can be defined via Duhamel's formula and if $\nabla h$ is $\mathcal{C}^1$ and locally Lipschitz, then the local flow is a well defined symplectomorphism \cite{kuksin2000analysis}. Under compactness assumptions on the nonlinearity, flow-maps can be approximated on bounded sets by finite dimensional symplectomorphisms. Specific examples of this type of equations are (see \cite{kuksin1995infinite} for more details): Nonlinear string equation in $\T$, $$\ddot u=u_{xx}+p(t,x,u)$$ where $p$ is a smooth function which has at most polynomial growth at infinity. Quadratic nonlinear wave equation in $\T^2$, $$\ddot u=\Delta u+a(t,x)u+b(t,x)u^2,$$
Nonlinear membrane equation on $\T^2$, $$\ddot u=-\Delta^2u+p(t,x,u),$$
Schr\"odinger equation with a convolution nonlinearity in $\T^n$,  $$-i\dot u=-\Delta u+V(x)u+\Big[\frac{\partial}{\partial \bar U}G(U,\bar U,t,x)\Big]*\xi,\quad U=u*\xi,$$ where $\xi$ if a fixed real function and $G$ is a real-valued smooth function.

For concreteness we will study the nonlinear string equation with bounded $\nabla h_t$ but the main result will still be true for the previous equations provided that the nonlinear part $\nabla h_t$ is sub-quadratic. 

Consider the periodic nonlinear string equation $$\ddot u=u_{xx}-f(t,x,u),\qquad u=u(t,x),$$ where $x\in \T=\R/ 2\pi \Z$ and $f$ is a smooth function which is bounded and has at most a polynomial growth in $u$, as well as its $u-$ and $t-$derivatives:$$\Big\lvert\frac{\partial ^a}{\partial u^a}\frac{\partial ^b}{\partial t^b}f(t,x,u)\Big\rvert\leq C_k(1+\lvert u\rvert)^{M_k},\quad for\quad\text{for}\quad a+b=k\quad \text{and all}\quad k\geq 0, $$ with $M_0=0$. Here $C_k$ and $M_k$'s are  non-negative constants. The hypothesis on $M_0$ is the one that will allow us latter to apply Theorem \ref{thm:coisotropic nonsqueezing} to the finite dimensional approximations. This hypothesis is verified by $f(t,x,u)=\sin u$ which gives the Sine-Gordon equation. Let us describe the Hamiltonian structure of this equation. We denote by $B$ the operator $B=(-\partial^2/\partial x^2+1)^{1/2}$ and remark that we may write the equation in the form \begin{align*}\dot u &= -Bv, \\\dot v &= (B-B^{-1})u+B^{-1}f(t,x,u).\end{align*} Define $E=E_+\times E_-=H^{\frac{1}{2}}(\T)\times H^{\frac{1}{2}}(\T)$ the product of Hilbert spaces where the scalar product of $H^{\frac{1}{2}}(\T)$ is given by $$\langle u_1,u_2\rangle =\frac{1}{2\pi}\int_0^{2\pi}Bu_1(x)u_2(x)dx.$$ Here $J(u,v)=(-v,u)$, the operator is $A=(B-B^{-1})\times B$ and $\nabla h_t(u,v)=(B^{-1}f(t,x,u),0)$ which has bounded norm over compact time intervals since $M_0=0$ by hypothesis. 

Let $\{\varphi^+_j\,|\,j\in \Z\}$ be the Hilbert basis of $E_+$ on which $B$ is diagonal given by the Fourier basis and denote by $\{\varphi^-_j:=-J\varphi_j^ +\,|\,j\in\Z \}$ the associated Hilbert basis of $E_-$. Moreover denote $E_k$ (resp. $E^k_+$ and $E^ k_-$) the Hilbert subspace generated by $\{\varphi^\pm_j\,|\, |j|\leq k\}$ (resp. $\{\varphi^+_j\,|\,|j|\geq k+ 1\}$ and $\{\varphi^-_j\,|\,|j|\geq k+ 1\}$) and $\Pi_k:E\rightarrow E_k$ (resp. $\Pi^k_+$ and $\Pi^ k_-$) the corresponding projection. We use the finite dimensional approximation together with Theorem \ref{thm:coisotropic nonsqueezing} to prove the following result:  

\begin{thm}\label{thm: infinite nonsqueezing}
	Denote by $\Phi^t:E\rightarrow E$ the flow of the nonlinear string equation satisfying the previous hypothesis. For every $k\in \N$, every compact subset $X$ of $E_k$ and every $t\in \R$ we have $$c(X)\leq \gamma(\Pi_k\Phi^t(X\times E^k_+)).$$
\end{thm}

As an example of what type of information we get from this theorem, consider the subspace $E_0$ which consists on constant functions. In this case Theorem \ref{thm: infinite nonsqueezing} gives us information on the global behavior of solutions with constant initial velocity provided that the projection on $E_0$ of the initial conditions is contained in a compact set $X$, for example in the closed ball of radius $r$. On the other hand, if we interchange the roles of $E^k_+$ and $E^k_-$ we get information about solutions whose initial position is given by a constant function. In particular we see that the energy of these solutions cannot be globally transfered to higher frequencies since the projection on $E_0$ cannot be contained in a ball $B_R^2$ with $R<r$.  

Theorem \ref{thm: infinite nonsqueezing} may also be seen as an existence result. We consider again the case $k=0$. By reordering the Hilbert basis we may project onto the symplectic plane of frequency $l$ of our choice. Suppose that $X$ is a ball of radius  $R$ in $\text{Vect}\{\varphi_l^+,\varphi_l^-\}\simeq \R^2$, we will ask if the projection is contained in a ball of radius $r$ in  $\text{Vect}\{\varphi_l^+,\varphi_l^-\}$. Let $U(t)=(u(t),v(t))\in E=H^{\frac{1}{2}}\times H^{\frac{1}{2}}$ be a solution of the nonlinear string equation, that is, such that $$\dot u=-B v\quad\text{and}\quad \ddot u=u_{xx}-f(t,x,u).$$ Use the symplectic Hilbert basis $\{\varphi_j^\pm\,|\,j\in \Z\}$ to write $$U(t)=\sum_ju_j(t)\varphi_j^++v_j(t)\varphi_j^-=\sum_j(u_j(t)-v_j(t)J)\varphi_j^+.$$ We will denote by $U_j(t)$ the complex number $u_j(t)-iv_j(t)$. Moreover denote by $E^0$ the Hilbert subspace of $E$ generated by $\{\varphi^\pm_j\,|\,|j|> 0\}$. If $0<r<R$ Theorem \ref{thm: infinite nonsqueezing} gives $\Phi^t( B_R\times E^0_+)\not\subseteq B_r\times E^0$ so we get:

\begin{cor}
	For any $l\geq 1$, any $R>r>0$ and any $t_0\in \R$ there exists a (mild) solution $U(t)=(u(t),v(t))$ of the nonlinear string equation in $H^{\frac{1}{2}}\times H^{\frac{1}{2}}$ such that $$v_j(0)=0\quad \text{for}\quad j\neq l\quad\text{and}\quad\lvert U_l(0)\rvert\leq R \quad \text{but}\quad  \lvert U_l(t_0)\rvert> r$$
\end{cor}

\section{The coisotropic camel: Viterbo's approach}\label{sec:lagcamelvit}
We provide here a proof of the Theorem \ref{thm:coisotropic nonsqueezing} which depends on Viterbo's spectral invariants, hence on generating functions instead of holomorphic curves. We start by proving the result for compactly supported Hamiltonian diffeomorphisms. 

\subsection{Generating functions and spectral invariants}
\paragraph{The classical setting.}
To a compactly supported Hamiltonian diffeomorphism $\psi$ of $\mathbb{R}^{2n}$ one associates a Lagrangian submanifold $L_\psi \subset T^*S^{2n}$ in the following way. Denote by $\overline{\mathbb{R}^{2n}}\times\mathbb{R}^{2n}$ the vector space $\mathbb{R}^{2n}\times\mathbb{R}^{2n}$ endowed with the symplectic form $(-\omega)\oplus\omega$. The graph $\Gamma(\psi):=\{(x,\psi(x));\,x\in \mathbb{R}^{2n}\}\subset \overline{\mathbb{R}^{2n}}\times\mathbb{R}^{2n}$ is a Lagrangian submanifold Hamiltonian isotopic to the diagonal $\Delta:=\{(x,x);\,x\in\mathbb{R}^{2n}\}\subset \overline{\mathbb{R}^{2n}}\times\mathbb{R}^{2n}$. Identifying $\overline{\mathbb{R}^{2n}}\times\mathbb{R}^{2n}$ and $T^*\mathbb{R}^{2n}$ via the symplectic isomorphism $$\mathcal{I}:(\bar{q},\bar{p},q,p)\mapsto(\frac{\bar{q}+q}{2},\frac{\bar{p}+p}{2},p-\bar{p},\bar{q}-q),$$ and noting that $\Gamma(\psi)$ and $\Delta$ coincide at infinity, we can produce a compact version of the Lagrangian submanifold $\Gamma(\psi)\subset\overline{\mathbb{R}^{2n}}\times\mathbb{R}^{2n}$, which is a Lagrangian sphere $L_\psi\subset T^*S^{2n}$. This Lagrangian submanifold $L_\psi$ is Hamiltonian isotopic to the 0-section and coincides with it on a neighbourhood of the north pole, so it has a generating function quadratic at infinity (called \gfqi in the following) by \cite{laudenbach1985persistance,sikorav1986immersions,sikorav1987intersections}. This is a function $S: S^{2n}\times \mathbb{R}^N\rightarrow \mathbb{R}$ which coincides with a non-degenerate quadratic form $Q:\mathbb{R}^N\rightarrow \mathbb{R}$ at infinity: $$\exists C>0\quad\text{such that}\quad S(x,\xi)=Q(\xi)\quad \forall x\in S^{2n}, |\xi|>C,$$ and such that $$L_\psi=\Big\{(x,\frac{\partial S}{\partial x}),\, (x,\xi)\in S^{2n}\times \mathbb{R}^{N},\,\frac{\partial S}{\partial \xi}(x,\xi)=0\Big\} \subset T^*S^{2n}.$$ 
with $0$ being a regular value of $(x,\xi)\mapsto \frac{\partial S}{\partial \xi}(x,\xi)$. A direct consequence of the definition is that $\Sigma_S:=\{(x,\xi)\,|\,\frac{\partial S}{\partial \xi}(x,\xi)=0\}$ is a submanifold and that the map $i_S:\Sigma_S\rightarrow T^*S^{2n}$ given by $(x,\xi)\mapsto(x,\frac{\partial S}{\partial x}(x,\xi))$ is an immersion. When $S$ is a \gfqi that generates an embedded submanifold we moreover ask that $i_S$ is a diffeomorphism between $\Sigma_S$ and $L_\psi$, so every $S$ has a unique critical point associated to $(N,0)$ given by $i_S^{-1}(N,0)$. Denote by $E^\lambda :=\{S\leq \lambda\}$, $i_\lambda:(E^\lambda,E^{-\infty}) \hookrightarrow (E^{+\infty},E^{-\infty})$, $H^*(E^{+\infty},E^{-\infty})\overset{T^{-1}}{\simeq}H^*(S^{2n})$ ($T$ is the Thom isomorphism). One can select spectral values $c(\alpha,S)$ for $\alpha \in H^*(S^{2n})$ by: 
$$
c(\alpha,S):=\inf\{\lambda\,|\,i_\lambda^* (T\alpha)\neq 0\}.
$$
The \gfqi associated to $\psi$ is unique up to certain explicit operations \cite{theret1999complete,viterbo1992symplectic} so there is a natural normalization (requiring that $S(i_S^{-1}(N,0))=0$) that ensures that the value $c(\alpha,S)$ does not depend on the \gfqi  , so we denote it henceforth $c(\alpha,\psi)$. It is a symplectic invariant in the sense that if $\Phi\in \text{Symp}(\mathbb{R}^{2n})$ then $c(\alpha,\Phi\circ \psi\circ\Phi^{-1})=c(\alpha,\psi)$. Taking for $\alpha$ generators $1$ and $\mu$ of $H^0(S^{2n})$ and $H^{2n}(S^{2n})$ respectively, we therefore get two spectral invariants $c(1,\psi)$ and $c(\mu,\psi)$ of Hamiltonian diffeomorphisms, and a spectral norm $\gamma(\psi):=c(\mu,\psi)-c(1,\psi)$. These invariants can be used in turn to define symplectic invariants of subsets of $\mathbb{R}^{2n}$. First if $U$ is an open and bounded set: 
\begin{gather}
c(U)=\sup\{c(\mu,\psi),\,\psi\in\text{Ham}^c(U)\}, \label{eq:cX}\\ \quad \gamma(U)=\inf\{\gamma(\psi),\; \psi\in \ham^c(\R^{2n}),\; \psi(U)\cap U=\emptyset\} \label{eq:gX} \end{gather}
 If $V$ is an open (not necessarily bounded) subset of $\R^{2n}$ we define $c(V)$ (resp. $\gamma(V)$) as the supremum of the values of $c(U)$ (resp. $\gamma(U)$) for all open bounded $U$ contained in $V$. If $X$ is an arbitrary domain of $\R^{2n}$ then we define its capacity $c(X)$ (resp. $\gamma(X)$) to be the infimum of all the values $c(V)$ (resp. $\gamma(V)$) for all open $V$ containing $X$.

\paragraph{Symplectic reduction \cite[\S 5]{viterbo1992symplectic}.}
Let us first state a general result for the spectral invariants of the reduction of some Lagrangian submanifolds. The first claim is proposition 5.1 in \cite{viterbo1992symplectic}. We include a proof for the sake of completeness. 
\begin{prop}\label{prop:redspec}
	Let $N$ and $B$ be two connected compact oriented manifolds, $S$ a \gfqi for a Lagrangian submanifold in $T^*(N\times B)$, $b$ a point in $B$ and $S_b:=S(\cdot,b,\cdot)$. Let $\alpha\in H^*(N)$ and $\mu_B\in H^*(B)$ the orientation class of $B$. Then,
	$$
	c(\alpha\otimes 1,S)\leq c(\alpha,S_b)\leq c(\alpha\otimes\mu_B,S).
	$$ 
	Moreover, if $\tilde K(x,b,\xi)=K(x,\xi)$ for all $(x,b,\xi)\in N\times B\times \mathbb{R}^N$, $c(\alpha\otimes 1,\tilde K)= c(\alpha,K)= c(\alpha\otimes\mu_B,\tilde K)$.
\end{prop}

\begin{proof} Let as before $E^\lambda:=\{S\leqslant \lambda\}$, and $E^\lambda_b:=\{S_b\leqslant \lambda\}$. Consider the commutative diagram 
	$$
	\xymatrix{
		H^*(N\times B)\ar[r]^{T}\ar[d] & H^*(E^\infty,E^{-\infty})\ar[r]^{i_\lambda^*} \ar[d] & H^*(E^\lambda,E^{-\infty})\ar[d]\\
		H^*(N)\ar[r]^{T} & H^*(E_b^\infty,E_b^{-\infty})\ar[r]^{i_\lambda^*} & H^*(E_b^\lambda,E_b^{-\infty})}
	$$
	where the map $H^*(N\times B)\rightarrow H^*(N)$ is induced by the injection $N\rightarrow N\times \{b\}\rightarrow N\times B$, and coincides with the composition of the projection on $H^*(N)\otimes H^0(B)$ and the obvious identification $H^*(N)\otimes H^0(B)\rightarrow H^*(N)$.  
	Since the diagram is commutative, $i_\lambda^*T(\alpha)\neq 0$ implies $i_\lambda^*T(\alpha\otimes 1)\neq 0$, so $c(\alpha\otimes 1,S)\leq c(\alpha,S_b)$. To get the second inequality, we need to introduce spectral invariants defined {\it via} homology. The Thom isomorphism is now $T:H_*(S^{2n})\overset{\sim}{\to} H_*(E^{+\infty},E^{-\infty})$, and
	$$
	c(A,S)=\inf\{\lambda\;|\; TA\in \im(i_{\lambda*})\}. 
	$$
	The homological and cohomological invariants are related by the equality $c(\alpha,S)=-c(\pd(\alpha),-S)$ \cite[Proposition 2.7]{viterbo1992symplectic}. In the homology setting, the commutative diagram becomes  
	$$
	\xymatrix{
		H_*(N\times B)\ar[r]^{T} & H_*(E^\infty,E^{-\infty}) & H_*(E^\lambda,E^{-\infty})\ar[l]_{i_{\lambda*}}\\
		H_*(N)\ar[u]\ar[r]^{T} & H_*(E_b^\infty,E_b^{-\infty}) \ar[u]& H_*(E_b^\lambda,E_b^{-\infty})\ar[u]\ar[l]_{i_{\lambda*}}
	}.
	$$
	As before, if $A\in H_*(N)$ verifies $T(A)\in \im(i_{\lambda*})$, then $T(A\otimes[b])\in \im(i_{\lambda*})$, so $c(A\otimes [b],S)\leqslant c(A,S_b)$ for all $A\in H_*(N)$ (and all $S$). Thus, 
	$$c(\alpha,S_b)=-c(\pd(\alpha),-S_b)\leqslant -c(\pd(\alpha)\otimes [b],-S)=-c(\pd(\alpha\otimes \mu_B),-S)$$ and $-c(\pd(\alpha\otimes \mu_B),-S)=c(\alpha\otimes \mu_B,S)$  so we get $c(\alpha,S_b)\leq c(\alpha\otimes \mu_B,S)$.
	Finally, if $\tilde{K}(x,b,\xi)=K(x,\xi)$ for all $(x,b,\xi)\in N\times B\times \mathbb{R}^N$ then $E^\lambda=E_b^\lambda \times B$ so 
	$i_\lambda^*(\alpha\otimes \beta)=(i_\lambda^*\alpha)\otimes \beta$. This gives $c(\alpha\otimes 1,\tilde K)=c(\alpha\otimes \mu_B,\tilde K)=c(\alpha,K)$.
\end{proof}
\begin{rem}\label{rk:gfqired} 
	To understand the context of the previous statement, notice that when a Lagrangian submanifold $L\subset T^*N\times T^* B$ has a \gfqi $S$, and has transverse intersection with a fiber $T^*N\times T^*_bB$ for some $b\in B$, the function $S_b$ is a \gfqi for the reduction $L_b$ of $L\cap T^*N\times T^*_bB$ (which is an immersed Lagrangian of $T^*N$).
\end{rem}

Following \cite[\S 5]{viterbo1992symplectic}, we work on $\mathbb{R}^{2m}\times T^*\mathbb{T}^k\simeq \mathbb{R}^{2m}\times \mathbb{R}^k\times \mathbb{T}^k$ endowed with coordinates $(z,p,q)$. Let $\pi:\mathbb{R}^{2m}\times \mathbb{R}^k\times \mathbb{R}^k\rightarrow \mathbb{R}^{2m}\times \mathbb{R}^k\times \mathbb{T}^k$ be the projection and consider a Hamiltonian diffeomorphism $\psi\in\text{Ham}^c(\mathbb{R}^{2m}\times T^*\mathbb{T}^k)$ with coordinates $(\psi_z,\psi_p,\psi_q)$ generated by $H_t$. It is easy to see that $H_t\circ \pi$ generates a lift $\tilde{\psi}\in \text{Ham}(\mathbb{R}^{2m}\times \mathbb{R}^k\times \mathbb{R}^k)$ such that $$\begin{cases}\tilde{\psi}_z(z,p,\tilde{q}+1)=\tilde{\psi}_z(z,p,\tilde{q})=\psi_z(z,p,q)\quad (\text{with } (z,p,q)=\pi(z,p,\tilde{q}))\\ \tilde{\psi}_p(z,p,\tilde{q}+1)=\tilde{\psi}_p(z,p,\tilde{q})=\psi_p(z,p,q)\\\tilde{\psi}_{\tilde{q}}(z,p,\tilde{q}+1)=\tilde{\psi}_{\tilde{q}}(z,p,\tilde{q})+1\end{cases}$$
Again, the graph of $\tilde{\psi}$ is a Lagrangian submanifold $\Gamma(\tilde{\psi})\subset \overline{\mathbb{R}^{2m}\times \mathbb{R}^{2k}}\times\mathbb{R}^{2m}\times\mathbb{R}^{2k}$ that under $\mathcal{I}$ becomes a Lagrangian submanifold of $T^*\mathbb{R}^{2m}\times T^*\mathbb{R}^{2k}$ whose points are denoted by $\Gamma_{\tilde{\psi}}(z,p,\tilde{q})$ equal to
$$\big(\mathcal{I}(z,\psi_z(z,p,q)),\frac{p+\psi_p(z,p,q)}{2},\frac{\tilde{q}+\tilde{\psi}_{\tilde{q}}(z,p,\tilde{q})}{2},\tilde{q}-\tilde{\psi}_{\tilde{q}}(z,p,\tilde{q}),\psi_p(z,p,q)-p\big).$$
Now $\tilde{\psi}_{\tilde{q}}(z,p,\tilde{q}+1)=\tilde{\psi}_{\tilde{q}}(z,p,\tilde{q})+1$ implies that $\Gamma_{\tilde{\psi}}$ descends to an embedding $\tilde{\Gamma}_\psi:\mathbb{R}^{2m}\times \mathbb{R}^k\times \mathbb{T}^k\rightarrow T^*\mathbb{R}^{2m}\times T^*\mathbb{R}^k\times T^*\mathbb{T}^k$, given by 
$$\tilde{\Gamma}_{\psi}(z,p,q)=\big(\mathcal{I}(z,\psi_z),\frac{p+\psi_p}{2},q-\psi_q,\frac{q+\psi_q}{2},\psi_p-p\big).$$
This embedding $\tilde{\Gamma}_\psi$ is Hamiltonian isotopic to the zero-section, and coincides with the zero-section at infinity. As in the classical situation, $\tilde{\Gamma}(\psi):=\im\tilde{\Gamma}_\psi$ can be compactified to a Lagrangian submanifold $$L_\psi\subset T^*(S^{2m}\times S^{k}\times \mathbb{T}^k),$$
which is Hamiltonian isotopic to the zero-section, and coincides with the zero-section on a neighbourhood of $\{N\}\times S^k\times \mathbb{T}^k$ and of $S^{2m}\times\{N\}\times \mathbb{T}^k$. After normalization (by $S(i^{-1}_S(N,N,0,0))=0$), the \gfqi of $L_\psi$ provides spectral invariants $c(\alpha\otimes\beta\otimes\gamma,\psi)$ for $\alpha\in H^*(S^{2m})$, $\beta\in H^*(S^k)$ and $\gamma \in H^*(\mathbb{T}^k)$. 
As in (\ref{eq:cX}), these invariants can be used to define $c(\alpha\otimes\beta\otimes\gamma,X)$, for subsets $X\subset \R^{2m}\times T^*\T^k$. 

\begin{prop}\label{first inequality}
	If $X$ is a compact subset of $\R^{2m}$ then 
	$$c(X)\leq c(\mu\otimes\mu\otimes 1,X\times \{0\}\times \mathbb{T}^k).$$
\end{prop}

\noindent{\it Proof:}
Let $\mathcal{U}$ be a bounded \nbd of $X$, $\phi\in \ham^c(\mathcal{U})$. By compactness of $X$ and by definition of the spectral capacities $c$, it is enough to find, for any \nbd $\mathcal{V}$ of $0$ in $\R^k$,  a $\Psi\in \ham^c(\mathcal{U}\times \mathcal{V}\times \T^k)$ such that $c(\mu,\phi)\leqslant c(\mu\otimes \mu\otimes 1,\Psi)$. 

Let $H:\mathbb{R}\times\mathbb{R}^{2m}\rightarrow \mathbb{R}$ be a generator of $\phi$, 
and $\chi \in C^{\infty}_c(\mathcal{V})$ with $\chi(0)=1$ and $\frac{\partial\chi}{\partial p}(0)=0$. The Hamiltonian $\chi H$ of $\mathbb{R}^{2m}\times\mathbb{R}^{k}\times \mathbb{T}^k$ generates a compactly supported Hamiltonian diffeomorphism that we will note $\Psi=(\psi_z,\psi_p,\psi_q)$. It is easy to see that 
$$
\Psi(z,p,q)=(\psi_z(z,p),p,q+C(z,p))
$$ 
with $C(z,p)=\int_0^1\frac{\partial \chi}{\partial p}(p)H(t,z)dt$, and that $C(z,0)=0$ and $\psi_z(z,0)=\varphi(z)$. The embedding $\Gamma_\Psi: \mathbb{R}^{2m}\times \mathbb{R}^k\times\mathbb{T}^k\rightarrow  T^*\mathbb{R}^{2m}\times T^* \mathbb{R}^k\times T^*\mathbb{T}^k$ it thus given by 
$$
\tilde{\Gamma}_\Psi(z,p,q)=(\mathcal{I}(z,\psi_z(z,p)),\,p,\,-C(z,p),\, q+\frac{1}{2}C(z,p),\,0).
$$ 
By definition, when we compactify $\text{Im}\, \tilde{\Gamma}_\Psi$ we get $L_\Psi$ which, by the previous expression, is easily seen to be transverse to $T^*S^{2m}\times T_{0}S^{k}\times T^*\mathbb{T}^k$. Now $$\tilde{\Gamma}_\Psi(z,0,q)=(\mathcal{I}(z,\varphi(z)),\,0,\,0,\, q,\,0).$$ so $L_\Psi\cap T^*S^{2m}\times T_{0}S^{k}\times T^*\mathbb{T}^k=L_\varphi\times\{(0,0)\}\times 0_{\mathbb{T}^k}$ and the reduction is $L_\varphi\times 0_{\mathbb{T}^k}$ which is also Hamiltonian isotopic to the zero-section. 
Therefore, by remark \ref{rk:gfqired}, if $S$ is a \gfqi for $L_\Psi$, $S_{0}$ is a \gfqi for $L_\varphi\times 0_{\mathbb{T}^k}$. On the other hand, if $K$ is a \gfqi for $L_\varphi$ then $\tilde{K}(z,q,\xi)=K(z,\xi)$ is also a \gfqi for $L_\varphi\times 0_{\mathbb{T}^k}$. Moreover, both $S_{0}$ and $\tilde{K}$ have $0$ as the critical value associated to $\{N\}\times \{q\}$, so by uniqueness of \gfqi $c(\mu\otimes 1,\tilde{K})=c(\mu\otimes 1,S_{0})$.  By proposition \ref{prop:redspec}, 
$$
c(\mu,K)=c(\mu\otimes 1,\tilde{K})=c(\mu\otimes 1,S_{0})\leq c(\mu\otimes\mu\otimes 1,S),
$$
which precisely means that $c(\mu,\varphi)\leq c(\mu\otimes\mu\otimes 1,\Psi)$.
\cqfd
The next proposition is a modified version of \cite[proposition 5.2]{viterbo1992symplectic}. Since the proof there is a bit elliptical, (it refers to the proofs of several other propositions of the same paper) we give more indications in section \ref{sec:proofviterbo} below. 

\begin{prop}\label{second inequality}
	Consider a compact set $Z\subset \mathbb{R}^{2m}\times\mathbb{R}^k\times \mathbb{T}^k$, a point $w\in\mathbb{T}^k$ and the reduction $Z_{w}:=(Z\cap\{q=w\})/\mathbb{R}^k$. Then $$c(\mu\otimes\mu\otimes 1,Z)\leq \gamma(Z_{w}).$$
\end{prop}

\subsection{Non-squeezing and symplectic reduction}
Together, propositions \ref{first inequality} and \ref{second inequality} provide the non-squeezing statement we are looking for. Recall that if $Z\subseteq \C^n$ and $W$ is a coisotropic subspace of $\C^n$ then the symplectic reduction of $Z$ is defined by $\Red_W(Z)=\pi_W(Z\cap W)$ where $\pi_W:W\rightarrow W/W^\omega$ is the natural projection.
\begin{thm}\label{thm:compact support}
	Let $X\subseteq \C^m$ be a compact subset, consider $X\times\mathbb{R}^{n-m}\subset \mathbb{C}^m\times \mathbb{C}^{n-m}$ and denote by $W:=\C^m\times i\R^{n-m}$. For every compactly supported Hamiltonian diffeomorphism $\psi$ of $\mathbb{C}^n$ we have $$c(X)\leq\gamma(\Red_W(\psi(X\times \R^{n-m}))).$$
\end{thm}
\begin{proof}
	Since $\psi$ has compact support, we can view it as a symplectomorphism of $\mathbb{C}^m\times T^*\mathbb{T}^{n-m}\simeq \mathbb{R}^{2m}\times\mathbb{R}^{n-m}\times\mathbb{T}^{n-m}$. In this setting $X\times\mathbb{R}^{n-m}$ 
	is seen as $X\times\{0\}\times \mathbb{T}^{n-m}$ and $\psi(X\times \{0\}\times \mathbb{T}^{n-m})_0$ coincides with $\Red_W(\psi_t(X\times \R^{n-m}))$. Now applying proposition \ref{first inequality}, invariance, proposition \ref{second inequality} and monotonicity we get the chain of inequalities: \begin{gather*}c(X)\leq c(\mu\otimes\mu\otimes 1,X\times \{0\}\times \mathbb{T}^{n-m})=c(\mu\otimes\mu\otimes 1,\psi(X\times \{0\}\times \mathbb{T}^{n-m}))\\
	\leq \gamma(\psi(X\times \{0\}\times \mathbb{T}^{n-m})_0)=\gamma(\Red_W(\psi(X\times \R^{n-m}))).
	\end{gather*}\end{proof}

\begin{cor}[Lagrangian camel theorem] Let $L:=S^1(r)^m\times\mathbb{R}^{n-m}\subset \mathbb{C}^m\times \mathbb{C}^{n-m}$ be a standard Lagrangian tube. Assume that there is a compactly supported Hamiltonian diffeomorphism $\psi$ of $\mathbb{C}^n$ such that $\psi(L)\cap (\mathbb{C}^m\times i\mathbb{R}^{n-m})\subset Z(R)\times i\mathbb{R}^{n-m}$ where $Z(R)$ is a symplectic cylinder of capacity $R$. Then $r\leq R$. \end{cor} 
\begin{proof}
	Theorem \ref{thm:compact support} gives $c(S^1(r)^m)\leqslant \gamma(Z(R))=\pi R^2$ and $c(S^1(r)^m)=\pi r^2$ by \cite[remark 1.5]{theret1999lagrangian}. 
\end{proof}

\subsection{Proof of proposition \ref{second inequality}}\label{sec:proofviterbo}
Let $Z\subset \R^{2m}\times \R^k\times \T^k$ and $Z_w:=\big(Z\cap \{q=w\}\big)/\R^k$. We need to show that $c(\mu\otimes\mu\otimes 1,Z)\leqslant \gamma(Z_w)$. 
Let $V\subset \R^{2m}$ be an arbitrary \nbd of $Z_w$, and 
$$
U:=(\mathbb{R}^{2m}\times\mathbb{R}^k\times \mathbb{T}^k\setminus\{w\})\cup (V\times\mathbb{R}^k\times\mathbb T^k).
$$
Obviously, $Z\subset U$ and $U_w=V$, so by monotonicity of $c$, it is enough to prove that $c(\mu\otimes\mu\otimes 1,U)\leqslant \gamma(U_w)$. Notice moreover that any Hamiltonian diffeomorphism of $\R^{2m}$ that displaces $V=U_w$ also displaces its filling 
$\tilde U_w:=\R^{2m}\priv F_w^\infty$,  where $F_w^\infty$ is the unbounded connected component of $\R^{2m}\priv U_w$. Thus $\gamma(U_w)=\gamma(\tilde U_w)$, so we may as well assume that $\R^{2m}\priv U_w$ is connected and unbounded, which we do henceforth.  Let $\psi\in\text{Ham}^c(U)$ and $\varphi\in\text{Ham}^c(\mathbb{C}^m)$ be such that $\varphi(U_{w})\cap U_{w}=\emptyset$. We need to prove that 
$$
c(\mu\otimes\mu\otimes 1,\psi)\leq \gamma(\varphi).
$$

We know that the Lagrangian submanifold $L_\varphi$ in $T^*S^{2m}$ is isotopic to the zero section by a Hamiltonian diffeomorphism $\Phi$ and has a \gfqi $K:S^{2m}\times \mathbb{R}^d\rightarrow \mathbb{R}$. This diffeomorphism $\Phi$ induces a Hamiltonian diffeomorphism $\tilde{\Phi}:=\Phi\times \id$ on $T^*S^{2m}\times T^*S^k$ that verifies $\tilde{\Phi}(0)=L_\varphi\times 0$ and 
$\tilde K(x,y,\eta):=K(x,\eta)$ (defined on $S^{2m}\times S^k\times \R^d$) is a \gfqi for this submanifold. Now for a \gfqi $S$ of $L_\psi$ we have
$$
c(\mu\otimes\mu\otimes 1,\psi)=c(\mu\otimes\mu\otimes 1,S)\leqslant c(\mu\otimes \mu,S_w)\leqslant c(\mu\otimes\mu,S_w-\tilde{K})-c(1\otimes 1,-\tilde{K}).
$$
The first inequality above follows from proposition \ref{prop:redspec}, while the second one is the triangle inequality for spectral invariants \cite[proposition 3.3]{viterbo1992symplectic} (because $(\mu\otimes \mu)\cup (1\otimes 1)=\mu\otimes\mu$). The following lemmas  ensure that $c(\mu\otimes \mu,S_w-\tilde{K})=c(\mu\otimes \mu,-\tilde{K})$, so (applying proposition \ref{prop:redspec}):
$$
c(\mu\otimes\mu\otimes 1,\psi)\leqslant c(\mu\otimes \mu,-\tilde{K})-c(1\otimes 1,-\tilde{K})=c(\mu,-K)-c(1,-K)=\gamma(\phi^{-1}),$$
and $\gamma(\phi^{-1}) =\gamma(\phi)$ which gives the desired inequality.

Consider a Hamiltonian path $\psi^t$ from the identity to $\psi$ in $\text{Ham}^c(U)$
and a Hamiltonian path $\Psi^t$ of $T^*S^{2m}\times T^*S^k\times T^*\mathbb{T}^k$ such that $\Psi^t(0)=L_{\psi^t}$.
This path gives rise to a family of \gfqi $S^t$, continuous in $t$, that generate $L_{\psi^t}$ for all $t$ and that coincide with a fixed quadratic form $Q$ outside a compact set independent of $t$ \cite{laudenbach1985persistance,sikorav1986immersions,sikorav1987intersections}. The first lemma ensures that we can further assume  $S_t$ normalized. 

\begin{lem}
	$G^t:=S^t-c(\mu\otimes 1\otimes \gamma,S^t)$ is a continuous family of normalized generating functions for $L_{\psi^t}$. Moreover there exists a family of fiber preserving diffeomorphisms $\varphi_t$ such that $G^t\circ \varphi_t$ is a continuous family of normalized \gfqi for $L_{\psi^t}$. 
\end{lem}
\begin{proof}
	To start with, we know that both functions $S^t_N(p,q,\xi)=S^t(N,p,q,\xi)$ and $S_N^t(z,q,\xi)=S^t(z,N,q,\xi)$ generate the zero sections so they have just one critical value. Moreover $S^t(i_{S^t}^{-1}(N,N,q,0))$ is a common critical value so they are both the same. Using proposition \ref{prop:redspec} we get $c( 1\otimes \gamma,S_N^t)\leq c(\mu\otimes 1\otimes \gamma,S^t)\leq c(\mu\otimes \gamma,S_N^t)$ so $c(\mu\otimes 1\otimes \gamma,S^t)=S^t(i_{S^t}^{-1}(N,N,q,0))$ determines continuously the critical value at infinity.
	
	For the second part, define $c_t:=c(\mu\otimes 1\otimes \gamma,S^t)$ and recall that $S^t$  equals  $Q$ outside a compact set. Let $\chi:\R^N\to [0,1]$ be a compactly supported function with $\chi\equiv 1$ in a \nbd of $0$, and $X_t(\xi):=(1-\chi(\xi))c_t\frac {\grad Q(\xi)}{\Vert\grad Q(\xi)\Vert^2}$, seen as an autonomous vector field ($t$ is not the parameter of integration). This vector field $X_t$ is well-defined and complete because $Q$ is non-degenerate, so $\phi_t:=\Phi_{X_t}^1$ is well-defined. Moreover, if $\xi$ lies far away in $\R^N$, $\Phi_{X_t}^r(\xi)$ remains on the set $\{1-\chi=1\}$ for all $r\in [0,1]$, so $Q\circ \Phi^r_{X_t}(\xi)=Q(\xi)+rc_t$. As a consequence, $(Q-c_t)\circ \phi_t=Q$ outside a compact set, so $G_t\circ \phi_t:=G_t(z,p,q,\phi_t(\xi))$ is a \gfqi for $L_{\psi^t}$. Since moreover $G_t$ is normalized, so is $G_t\circ \phi_t$. Finally, the family $\phi_t$ is obviously continuous in the $t$ variable. 
\end{proof}

\begin{lem}\label{constant}
	Let $S^t$ be a continuous family of normalized \gfqi for the Lagrangian $L_{\psi^t}$. Then
	$c(\mu\otimes\mu,S^t_w-\tilde{K})$ is a critical value of $-\tilde{K}$ and 
	as a consequence $c(\mu\otimes \mu,S_w-\tilde{K})=c(\mu\otimes \mu,-\tilde{K})$.
\end{lem}
\begin{proof}
	Recall that points in $L_{\psi^t}$ are of the form 
	$$
	\tilde{\Gamma}_{\psi^t}(z,p,q)=\big(\mathcal{I}(z,\psi_z^t),\frac{p+\psi_p^t}{2},q-\psi_q^t,\frac{q+\psi^t_q}{2},\psi^t_p-p\big)
	$$ 
	plus other points on the zero section that come from compactifying. Moreover, the functions $S^t_w$ formally generate the sets of points  
	$$
	\big(\mathcal{I}(z,\psi_z^t),\frac{p+\psi^t_p}{2},q-\psi^t_q\big)\quad\text{ for points $(z,p,q)$ that verify }\quad \frac{q+\psi^t_q}{2}=w,
	$$
	plus other points in the zero section. This set is denoted henceforth $L^t_w$. 
	Recall that the notation $S^t_w-\tilde K$ stands for the function $(z,p,\xi,\eta)\mapsto S^t(z,p,w,\xi)-K(z,\eta)$. 
	It is enough to prove that all critical points $(z,p,\xi,\eta)$ of $S^t_w-\tilde K$ are such that $(z,\eta)$ is a critical point of $-\tilde K$, while $(z,p,\xi)$ is a critical point of $S^t_w$ with critical value $0$. Letting $x:=(z,p)$, such a critical point verifies 
	$$
	\frac{\partial S^t_w}{\partial x}=\frac{\partial \tilde K}{\partial x}\quad\text{ and }\quad\frac{\partial S^t_w}{\partial\xi}=\frac{\partial \tilde K}{\partial \eta}=0,
	$$
	so it is associated to an intersection point of $L_w^t$ and $L_\phi\times 0$ in the fiber of $(z,p)$. This intersection point therefore verifies:
	$$
	q-\psi_q^t=0\quad\text{ and }\quad \frac{q+\psi^t_q}{2}=w \quad(\text{so } q=\psi^t_q=w),
	$$ 
	or will be on the zero section coming from critical points of $S^t$ at infinity. We claim that such a point of intersection must lie on $\mathcal{I}(U_w\times U_w)^c\times T^*S^k$. Indeed, if $\mathcal{I}(U_w\times U_w)\times T^*S^k \cap (L_\varphi\times 0) \neq \emptyset$, then $\Phi^{-1}(\mathcal{I}(U_w\times U_w))\cap 0\neq \emptyset$.
	But $\Phi^{-1}(\mathcal{I}(U_w\times U_w))=\mathcal{I}(\varphi^{-1}(U_w)\times U_w)$ does not intersect the zero section because $\varphi$ displaces $U_w$. 
	This in turn implies that the intersection point is on the zero section: if a point of $L^t_w$ is in $\mathcal{I}(U_w\times U_w)^c\times T^*S^k$, $(z,\psi_z^t)\in (U_w\times U_w)^c$ so $z\notin U_w$ or $\psi_{z}^t\notin U_w$. In both cases, $\psi^{t}(z,p,w)=(z,p,w)$ because $q=\psi_q^t=w$, and $\psi^t$ has support in $U$, which intersects $\{q=w\}$ along $U_w\times \R^k$.
	Thus, the point $\tilde \Gamma_{\psi^t}(z,p,w)$ is on the zero section, $(z,p,w,\xi)$ is indeed a critical point of $S_t$ and as a consequence $(z,\eta)$ is a critical point of $-\tilde K$. 
	In addition $(z,p,w)$ is in $U^c$ because $z\notin U_w$.
	
	Now we prove that all the points in $U^c$ have critical value $0$. Since $\supp \psi_t\Subset U$ and $U^c$ is  connected, there is an open connected  set $W$ that contains $U^c$ and that does not intersect $\supp \psi_t$ (for all $t$). Then $0_{W}\subset L_t$ so if $j:W\hookrightarrow L_t$ is the inclusion on the zero section, $f:=i_{S^t}^{-1}\circ j: W\rightarrow \Sigma_{S^t}$ is an embedding into the set of critical points. The open set $W$ is connected so $S^t\circ f$ is constant and all the points in $W$ have the same critical value. The fact that $S^t$ is normalized now implies that this value is zero. 
	
	Finally, Sard's theorem ensures that the set of critical values of $-\tilde{K}$ has measure zero, so it is  totally disconnected.  By continuity of the invariants, $c(\mu\otimes\mu,S^t_w-\tilde{K})$ is therefore constant, so $c(\mu\otimes \mu,-\tilde{K})=c(\mu\otimes\mu,S^1_w-\tilde{K})$.
\end{proof}

\subsection{Proof of the sub-quadratic case}
We proceed with the proof of Theorem \ref{thm:coisotropic nonsqueezing}. We reduce the sub-quadratic case to the compactly supported case and then use Theorem \ref{thm:compact support} to conclude. Note that $H$ is sub-quadratic if and only if for every $\epsilon>0$ there is an $A_\epsilon\geq 0$ such that $\lvert\nabla H_t(z)\rvert\leq A_\epsilon+ \epsilon\lvert z\rvert$. The following proposition implies that for sub-quadratic $H$ the map $\psi_t^H$ verifies the coisotropic non-squeezing property for every $t\in\R$.

\begin{prop}\label{prop: less than linear}
	Let $H_t$ be a Hamiltonian of $\C^n$ such that $\lvert\nabla H_t(z)\rvert\leq A+ B\lvert z\rvert$. Let $X\subset \C^k$ be a compact subset and consider the coisotropic subspace $W=\C^k\times i\R^{n-k}$. Then the flow $\psi_t$ of $H$ verifies $$c(X)\leq \gamma(\Red_W(\psi_t(X\times\R^{n-k})))$$ for every $|t|< \frac{ln 2}{3B}$. 
\end{prop}
\begin{proof}
	By considering the Hamiltonian $\frac{1}{B}H_{\frac{t}{B}}$ we may suppose $B=1$. Using Gronwall's lemma we get the inequalities $$ \lvert\psi_s(z)\rvert\leq e^s(\lvert z\rvert+A)-A\quad \text{and}\quad \lvert \psi_s(z)-z\rvert\leq (e^s-1)(\lvert z\rvert+A).$$ Suppose that $z\in\C^n$ verifies for a fixed $t\in \R$ $$z\in X\times \R^{n-m}\quad\text{and}\quad \psi_t(z)\in \C^m\times i\R^{n-m}.$$ We call such a $z$ a camel point and $\psi_{[0,t]}(z)$ a camel trajectory. Denote $\pi^m_+$ the natural projection on $\R^{n-m}$ of coordinates $(q_{m+1},\dots,q_n)$. Using the fact $X$ is contained in a ball of a certain radius $r$ and that $\pi^m_+\psi_t(z)=0$ we find $$\lvert z\rvert\leq r+ |\pi_+^m(z)|=r+ |\pi_+^m(\psi_t(z)-z)|\leq r+(e^t-1)(\lvert z\rvert+A).$$  In particular we see that if $e^t<2$, so if $t<\ln 2$, the camel points verify $$\lvert z\rvert\leq \frac{r+A}{2-e^t}.$$ Using the inequalities at the beginning of the proof we see that the set of camel trajectories is contained in a ball of radius $C=C(t,A,r)$ centered at the origin. The idea now is to build from $\psi_t$ a compactly supported Hamiltonian diffeomorphism $\varphi_t$ that coincides with $\psi_t$ on $B(0,R)$ for some $R>C$ and whose camel trajectories are also contained in this ball for $|t|\leq t_0$ for some $t_0>0$. Then the camel trajectories of both flows coincide so we can apply Theorem \ref{thm:compact support} for $\varphi_t$ to get the desired result. 
	
	Let $\chi:\R\rightarrow \R$ be a smooth function with values on $[0,1]$ that equals $1$ over the interval $[0,R]$, vanishes over $[2R,+\infty[$ and such that $\lvert \chi'\rvert\leq 2/R$. Note that on the support of $\chi'$ we have $|z|\leq 2R$ so $$|\chi'(|z|)|\leq \frac{2}{R}\leq \frac{4}{|z|}.$$  Define $G_t(z)=\chi(\lvert z\rvert)H_t(z)$ (the value of $R$ will be chosen later). It is a compactly supported function that generates a Hamiltonian diffeomorphism $\varphi_t$. Since we may suppose that $H_s(0)=0$ for all $s$ we have $|H_s(z)|\leq A|z|+\frac{|z|^2}{2}$.  We have $$\lvert\nabla G_s(z)\rvert=\lvert \chi'(\lvert z\rvert)\frac{z}{\lvert z\rvert}H_s(z)+\chi(\lvert z\rvert)\nabla H_s(z)\rvert\leq 4A+2|z|+A+|z|\leq 5A+3|z|$$ and the bound does not depend on $R$. In particular, by the same arguments as above, if $|t|\leq \frac{\ln 2}{3}$ then the camel trajectories of $\varphi_t$ are bounded by a constant independent of $R$. Choose $R$ big enough to contain the camel trajectories of $\psi_t$ and $\varphi_t$ and the proposition follows.  \end{proof}

The time bound in Proposition \ref{prop: less than linear} is not optimal and one may get a better one modifying the bound for $|\chi'|$, but this bound cannot be extended much more since the statement fails for bigger $t$  (see Remark \ref{rem: squeezing}).

\section{Hamiltonian PDEs}
Let $E$ be a real Hilbert space. A (strong) symplectic form on a real Hilbert space is a continuous 2-form $\omega:E\times E\rightarrow \R$ which is non-degenerate in the sense that the associated linear mapping $$\Omega:E\rightarrow E^*\quad  \text{ defined by }\quad \xi\mapsto \omega(\xi,\cdot)$$ is an isomorphism. Let $H:E\rightarrow \R$ be a smooth Hamiltonian function. In the same way as in the finite dimensional case one can define the vector field $X_H(u)=\Omega^{-1}(dH(u))$ and consider the ODE $$\dot{u}=X_H(u).$$ The situation  encountered in examples is however a little bit different. In most cases the Hamiltonian $H$ is not defined on the whole space $E$ but only on a dense Hilbert subspace $D_H(E)\subseteq E$. This raises the question of what a solution is and how to construct it.

\subsection{Semilinear Hamiltonian equations}

Denote by $\langle \cdot,\cdot\rangle$ the scalar product of $E$. Consider an anti-self-adjoint isomorphism $\bar J:E\rightarrow E$ and supply $E$ with the strong symplectic structure $$\omega(\cdot,\cdot)=\langle \bar J\cdot,\cdot\rangle.$$ 
Denote $J=(\bar J)^{-1}$ which is also a skew adjoint isomorphism of $E$. Take a possibly unbounded linear operator $A$ with dense domain such that $JA$ generates a $C^0$ group of (symplectic) transformations $$\{  e^{tJA} \,|\, t\in \R\} \quad\text{ with }\quad \lVert e^{tJA}\rVert_E\leq Me^{N|t|}$$ and consider the Hamiltonian function $$H_t(u)=\frac{1}{2}\langle Au,u\rangle+h_t(u),$$ where $h:E\times \R\rightarrow \R$ is smooth. The corresponding Hamiltonian equation has the form $$\dot u=X_H(u)=JAu+J\nabla h_t(u).$$ In this case the domain of definition of the Hamiltonian vector field is the same as the domain $D(A)$ of $A$ which is a dense subspace of $E$. This implies that classical solutions can only be defined on $D(A)$. More precisely by a \textit{classical solution} we mean a function $u:[0,T[\rightarrow E$ continuous on $[0,T[$, continuously differentiable on $]0,T[$, with $u(t)\in D(A)$ for $0<t<T$ and such that the equation is satisfied on $[0,T[$. 
Nevertheless the boundedness of the exponential allows us to define solutions in the whole space $E$ via Duhamel's formula:

\begin{defin} A continuous curve $u(t)\in \cc([0,T];E)$ is a \textit{(mild) solution} of the Hamiltonian equation in $E$ with initial condition $u(0)=u_0$ if for $0\leq t\leq T$,
	$$u(t)=e^{tJA}u_0+\int_0^t e^{(t-s)JA}J\nabla h_s(u(s))ds.$$
\end{defin}
One can easily verify that if $u(t)$ is a classical solution, then it is also a mild solution. For semilinear equations we know (see for example \cite{pazy1983semigroups}) that if $\nabla h$ is locally Lipschitz continuous, then for each initial condition there exists a unique solution which is defined until blow-up time. If moreover $\nabla h$ is continuously differentiable then the solutions with $u_0\in D(A)$ are classical solutions of the initial value problem. Locally we get a smooth flow map $\Phi_t:\mathcal{O}\subseteq E\rightarrow E$ defined on an open set $\mathcal{O}$. If every solution satisfies an a priory estimate $$\lVert u(t)\rVert_E\leq g(t,u(0))<\infty$$ where $g$ is a continuous function on $\R\times E$, then all flow maps $\Phi_t:E\rightarrow E$ are well defined and smooth. This is the case for example if $\lVert\nabla h_t(u)\rVert_E\leq C$. Remark that the choice of the linear map $A$ is arbitrary. Indeed if $JA$ generates a continuous group of transformations and $B$ is a bounded linear operator then $J(A+B)$ is an infinitesimal generator of a group $e^{tJ(A+B)}$ on $E$ satisfying $\lVert e^{tJ(A+B)}\rVert_E\leq Me^{N+M\lVert B\rVert|t|}$. One can then consider the linear part $J(A+B)$ and set $J\nabla h_t-JB$ as the nonlinear part. This indeterminacy is only apparent: classical solutions verify Duhamel's formula for $JA$ and $J(A+B)$ so both flow maps coincide over the dense subspace $D(A)$ which by continuity implies that the two flows are equal. 

\subsection{Nonlinear string equation}
Consider the periodic nonlinear string equation $$\ddot u=u_{xx}-f(t,x,u),\qquad u=u(t,x),$$ where $x\in \T=\R/ 2\pi \Z$ and $f$ is a smooth function which is bounded and has at most a polynomial growth in $u$, as well as its $u-$ and $t-$derivatives:$$\Big\lvert\frac{\partial ^a}{\partial u^a}\frac{\partial ^b}{\partial t^b}f(t,x,u)\Big\rvert\leq C_k(1+\lvert u\rvert)^{M_k},\quad for\quad\text{for}\quad a+b=k\quad \text{and all}\quad k\geq 0, $$ with $M_0=0$. Here $C_k$ and $M_k$'s are non-negative constants.  We now describe the Hamiltonian structure of this equation. Denote by $B$ the operator $B=(-\partial^2/\partial x^2+1)^{1/2}$ and remark that we may write the equation in the form \begin{align*}\dot u &= -Bv, \\\dot v &= (B-B^{-1})u+B^{-1}f(t,x,u).\end{align*} Define $E=H^{\frac{1}{2}}(\T)\times H^{\frac{1}{2}}(\T)$ as the product of Hilbert spaces where the scalar product of $H^{\frac{1}{2}}(\T)$ is given by $$\langle u_1,u_2\rangle =\frac{1}{2\pi}\int_0^{2\pi}Bu_1(x)u_2(x)dx.$$ If we define the function $$h_t(u,v)=-\frac{1}{2\pi}\int_0^{2\pi}F(t,x,u(x))dx,\qquad F=\int_0^ufdu.$$ we get $$\nabla h_t(u,v)=(B^{-1}f(t,x,u(x)),0).$$ The gradient verifies $\lVert \nabla h_t\rVert_E \leq C_0$. The polynomial growth condition on $f$ guarantees that there exists a $0<\theta<1/2$ such that $\nabla h$ has a $\mathcal C^1$ extension to $H^{\frac{1}{2}-\theta}(\T)\times H^{\frac{1}{2}-\theta}(\T)$. Moreover this implies that $\nabla h$ is locally Lipschitz in $E$ over compact time intervals (see \cite{kuksin1995infinite} for details). A special case where such properties are verified is $f(t,x,u)=\sin u$ which corresponds to the Sine-Gordon equation. In this case $\lVert \nabla h_t\rVert_E \leq 1$. Now putting $A=(B-B^{-1})\times B$ and defining $J:E\rightarrow E$ by $J(u,v)=(-v,u)$ we can write the nonlinear string equation as the semilinear PDE: $$(\dot u,\dot v)=JA(u,v)+J\nabla h_t(u,v).$$ Consider the symplectic Hilbert basis $\{\varphi_j^{\pm}\,|\,j\in \Z\}$ where $$\varphi_j^+=\frac{1}{(j^2+1)^{\frac{1}{4}}}(\varphi_j(x),0),\quad \varphi_j^-=\frac{1}{(j^2+1)^{\frac{1}{4}}}(0,-\varphi_j(x)),$$ with 
\[
\varphi_j(x)=\begin{cases}
\sqrt{2}\sin jx,\quad j>0,\\
\sqrt{2}\cos jx,\quad j\leq 0.
\end{cases}
\]
In this basis we have $(B\times B)\varphi_j^\pm=\sqrt{j^2+1}\varphi_j^\pm$ so if we denote $\lambda_j=\sqrt{j^2+1}$ we get that $$A\varphi_j^+=(\lambda_j-\frac{1}{\lambda_j})\varphi_j^+\quad\text{ and }\quad A\varphi_j^-=\lambda_j\varphi_j^-.$$
Now remark that $JA$ has eigenvalues $\{\pm i \sqrt{\lambda_j^2-1}=\pm ij\}$. If we calculate $e^{tJA}$ we get that its action on each symplectic plane $\varphi_j^+\R\oplus\varphi_j^-\R$ is given by the matrix 
$$\begin{pmatrix}
\cos tj & -\frac{\sqrt{j^2+1}}{j}\sin tj \\
\frac{j}{\sqrt{j^2+1}}\sin tj & \cos tj
\end{pmatrix}$$
which gets closer and closer to a rotation as $j$ goes to infinity. In particular we get a bounded group of symplectic linear maps. We conclude that for all $t\in \R$ the time $t$ map of the flow of the nonlinear string equation $\Phi_t:E\rightarrow E$ is defined on the whole space $E$.

\subsection{Finite dimensional approximation}
In this subsection we will follow \cite{kuksin1995infinite} for the particular case of the nonlinear string equation. We include the proofs for completeness. Recall that the Hilbert basis of $E$ is $\{\varphi_j^{\pm}\,|\, j\in \Z\}$ and denote $E_n$ the vector space generated by $\{\varphi_j^{\pm}\,|\, |j|\leq n\}$. It is a real vector space isomorphic to $\R^{2n+2}$. Let $E^n$ be the Hilbert space with basis $\{\varphi_j^{\pm}\,|\,  |j|> n\}$ so that $E=E_n\oplus E^n$ and write $u=(u_n,u^n)$ for an element $u\in E$. The fact that $J$ and $A$ preserve $E_n$ for all $n$ will allow us to define the finite dimensional approximations just by projecting the vector field. Let $\Pi_n:E\rightarrow E_n$ be the natural projection and consider the Hamiltonian function $$H_n(u)=\frac{1}{2}\langle Au,u\rangle+h_n(u)\quad\text{where}\quad h_n(u):=h_t(\Pi_n(u)).$$ The Hamiltonian equation now becomes $$\dot u=X_{H_n}(u)=JAu+J\nabla h_n(u),$$ where $\nabla h_n(u)=\Pi_n(\nabla h_t(\Pi_n(u)))$. Since $\nabla h_n$ continues to be locally Lipschitz and bounded, $X_{H_n}$ generates a global flow $\Phi_n^t$. This flow can be decomposed as $\Phi_n^t=e^{tJA}\circ V_n^t$ with $V_n^t(u)=(\phi_n^t(u_n),u^n)$. Here $\phi_n^t$ is a finite dimensional Hamiltonian flow on $E_n$ generated by the time dependent function $h_n\circ e^{tJA}$. We remark that this function has a bounded gradient so $\phi_n^t$ verifies Theorem \ref{thm:coisotropic nonsqueezing} for every $t\in \R$. The key point of the approximation is the following lemma which is a slight modification of a lemma in \cite[appendix 2]{kuksin1995infinite}:

\begin{lem}\label{lemma:epsilon} Denote $F_\theta= H^{\frac{1}{2}-\theta}(\T)\times H^{\frac{1}{2}-\theta}(\T)$ and let $K$ be a compact subset of $F_\theta$. Let $g:\R\times F_\theta\rightarrow E$ be a continuous map and fix a $T>0$. Then $$\sup_{(t,u)\in[-T,T]\times K}\lVert g_t(u)-g_t(\Pi_nu)\rVert_{E}$$ converges to zero as $n$ goes to infinity. Moreover, for every $R>0$ there exists a decreasing function $\epsilon_R:\N\rightarrow \R$ such that $\epsilon_R(n)\rightarrow 0$ as $n\rightarrow \infty$ and $$\lVert \nabla h_t(u)-\nabla h_n(u)\rVert_E \leq \epsilon_R(n)$$ for every $u\in B(0,R)$ and $|t|\leq T$.
\end{lem}
\begin{proof}
	By contradiction suppose that there is a sequence $\{(s_n,z_n)\}\subset [-T,T]\times K$ such that $\lVert g_{s_n}(z_n)-g_{s_n}(\Pi_nz_n)\rVert_E\geq\delta>0$ for every $n\in \N$. By compactness we may suppose that there is a converging subsequence $(s_{n_k},z_{n_k})\rightarrow (s,z)$. This sequence will also verify $\Pi_{n_k}z_{n_k}\rightarrow z$. We have $$\lVert g_{s_{n_k}}(z_{n_k})-g_{s_{n_k}}(\Pi_{n_k}z_{n_k})\rVert_E\leq \lVert g_{s_{n_k}}(z_{n_k})-g_s(z)\rVert_E+\lVert g_s(z)-g_{s_{n_k}}(\Pi_{n_k}z_{n_k})\rVert_E$$ and the quantity of the rhs converges to zero as $n_k$ goes to infinity by continuity of $g$. In particular, for $n_k$ big enough we get $\lVert g_{s_{n_k}}(z_{n_k})-g_{s_{n_k}}(\Pi_{n_k}z_{n_k})\rVert_E<\delta$, a contradiction.
	
	For the second claim recall that $\nabla h_t$ has an extension to $F_\theta$ for $\theta$ small enough (see \cite{kuksin1995infinite}). Denote by $\nabla \tilde h_t$ the extension and let $i:E \rightarrow F_\theta$ be the compact inclusion so that $\nabla h_t(u)=\nabla \tilde h_t(i(u))$. Recall that $\nabla h_n(u)=\Pi_n \nabla h_t(\Pi_n(u))$. We have $$\lVert \nabla h_t(u)-\nabla h_n(u)\rVert_E \leq\lVert \nabla h_t(u)-\Pi_n\nabla h_t(u)\rVert_E +\lVert \Pi_n \nabla h_t(u)-\Pi_n\nabla h_t(\Pi_nu)\rVert_E $$ $$\leq  \lVert \nabla \tilde h_t(i(u))-\Pi_n\nabla \tilde h_t(i(u))\rVert_E +\lVert \nabla \tilde h_t(i(u))-\nabla \tilde h_t(\Pi_ni(u))\rVert_E.$$ For every $R>0$ the sets $\bigcup_{|t|\leq T}\nabla \tilde h_t(i(B_E(0,R)))$ and $i(B_E(0,R)))$ are precompact in $E=F_0$ and $F_\theta$ respectively, so we may take the $\sup$ in $B_E(0,R)$ and $|t|\leq T$ and apply the first part of the lemma to conclude.
\end{proof}
Now we have all the tools we need for the finite dimensional approximation.

\begin{prop}[\cite{kuksin1995infinite}]\label{prop: approximation}
	 Fix a $t\in \R$. For each $R>0$ and $\epsilon>0$ there exists an $N$ such that if $n\geq N$ then $$\lVert V^t(u)-V^t_n(u)\rVert_E\leq \epsilon$$ for all $u\in B(0,R)$.
\end{prop}
\begin{proof}
	Duhamel's formula and the fact that that $e^{tJA}$ is a bounded operator give $$\lVert V^t(u)-V^t_n(u)\rVert_E\leq C \int_0^t\lVert \nabla h_s(\Phi^s(u))-\nabla h_n(\Phi^s_n(u))\rVert_Eds\leq$$$$\leq C\int_0^t \lVert \nabla h_s(\Phi^s(u))-\nabla h_s(\Phi^s_n(u))\rVert_Eds+C\int_0^t\lVert \nabla h_s(\Phi_n^s(u))-\nabla h_n(\Phi^s_n(u))\rVert_Eds.$$ If $u\in B_E(0,R)$ and $s\in [0,t]$ then $\lVert \nabla h\rVert_E$ bounded implies that for all $n\in \N$ the element $\Phi_n^s(u)$ wont leave a ball of radius $R'(R,t)$. We can now use Lemma \ref{lemma:epsilon} and the fact that $\nabla h$ is locally Lipschitz to get $$\lVert V^t(u)-V^t_n(u)\rVert_E\leq \tilde C\int_0^t \lVert V^s(u)-V^s_n(u)\rVert_E ds+Ct\epsilon(n).$$ By Gronwall's lemma we conclude that $$\lVert V^t(u)-V^t_n(u)\rVert_E\leq \epsilon(n)C(t)$$ where $C(t)$ depends continuously on $t$. The function $\epsilon(n)$ is decreasing and converges to zero so there exists an $N\in\N$ such that if $n\geq N$ then $\epsilon(n)C(t)\leq \epsilon$ which gives the result.
\end{proof}

\subsection{Coisotropic camel}
We now move towards the proof of Theorem \ref{thm: infinite nonsqueezing}. Recall that to state Theorem \ref{thm:coisotropic nonsqueezing} we had to divide the simplectic phase space into two Lagrangian subspaces that determine the coisotropic subspaces that we work with. In the infinite dimensional case we have $E=E_+\oplus E_-=H^{\frac{1}{2}}\times H^{\frac{1}{2}}$ where $E_+$ (resp. $E_-$) is generated by $\{\varphi^+_j\,|\,j\in \Z\}$ (resp. $\{\varphi^-_j\,|\,j\in\Z\}$). Moreover denote $E_k$ (resp. $E^k_+$ and $E^ k_-$) the Hilbert subspace generated by $\{\varphi^\pm_j\,|\,|j|\leq k\}$ (resp. $\{\varphi^+_j\,|\,|j|\geq k+ 1\}$ and $\{\varphi^-_j\,|\,|j|\geq k+ 1\}$) and $\Pi_k:E\rightarrow E_k$ (resp. $\Pi^k_+$ and $\Pi^ k_-$) the corresponding projection. First, lets state the infinite dimensional version of Theorem \ref{thm:coisotropic nonsqueezing}.

\begin{prop}\label{prop: infinite coisotropic}
	Fix a $k\geq 1$ and let $X$ be a compact set contained in $E_k$. Define $$C=\{u\in E\,|\, \Pi_ku\in X\text{ and }\Pi^k_-u=0\}.$$ Then for every $t\in \R$ we have $$c(X)\leq \gamma(\Pi_k(V^t(C)\cap\{\Pi_+^k=0\})).$$
\end{prop}

This is not a statement about the actual flow of the nonlinear string equation. Nevertheless using the fact that $e^{tJA}$ restricts to a symplectic isomorphism on each $E_n$ we get Theorem \ref{thm: infinite nonsqueezing}:
\begin{proof}[Proof of Theorem \ref{thm: infinite nonsqueezing}]
	We always have the inclusion $\Pi_k(V^t(C)\cap\{\Pi_+^k=0\})\subseteq \Pi_kV^t(C)$ so by Proposition \ref{prop: infinite coisotropic} and monotonicity of the symplectic capacity $\gamma$ we have $$c(X)\leq \gamma(\Pi_k(V^t(C)\cap\{\Pi_+^k=0\}))\leq\gamma(\Pi_kV^t(C)).$$ The linear operator $e^{tJA}$ restricts to a symplectic isomorphism on each $E_n$ which commutes with $\Pi_k$ and the capacity $\gamma$ is invariant under symplectic transformations so $$\gamma( \Pi_kV^t(C))=\gamma( e^{-tJA}\Pi_ke^{tJA}V^t(C))=\gamma(\Pi_k\Phi^t(C)).$$ which gives the desired result.
\end{proof}

The proof of proposition \ref{prop: infinite coisotropic} relies on the finite dimensional result and it is the finite dimensional approximation of the flow that allows us to go from finite to infinite dimensions. For these reasons we start with the following lemma:

\begin{lem}\label{lemma: finite-infinite}
	\begin{enumerate}
		\item Fix a $k\geq 1$ and let $X$ be a compact set contained in $E_k$. Then for every $t\in \R$ and $n>k$ we have $c(X)\leq \gamma(\Pi_k(V_n^t(C)\cap\{\Pi_+^k=0\})).$
		\item The set $\cup_n\{u\in C\,|\,\Pi_+^kV_n^tu=0\}\subseteq E$ is bounded by a constant $R(t)$.
		\item The set $\{u\in C\,|\,\Pi_+^kV^tu=0\}$ is compact and so is $V^t(C)\cap\{\Pi_+^k=0\}$.
		
	\end{enumerate}
\end{lem}
\begin{proof}
	Recall that $V_n^tu=(\phi_n^t(u_n),u^n)$ where $\phi_n^t$ is a finite dimensional flow generated by a sub-quadratic Hamiltonian function so it verifies Theorem \ref{thm:coisotropic nonsqueezing}. An easy computation shows that $V_n^t$ verifies the statement if and only if $\phi_n^t$ verifies Theorem \ref{thm:coisotropic nonsqueezing} on $E_n$.
	
	For the second claim let $u\in E$ and decompose its norm as $\lVert u\rVert\leq\lVert \Pi_ku\rVert+\lVert \Pi_+^ku\rVert+\lVert \Pi_-^ku\rVert$. If $u\in C$ then by definition $\Pi_ku$ belongs to $X$ which is compact contained in a ball of a certain radius $r$ and $\Pi_-^ku=0$ so $\lVert u\rVert\leq r+\lVert \Pi_+^ku\rVert$. It is then enough to show that $\Pi_+^kV_n^tu=0$ implies $\lVert\Pi_+^ku\rVert\leq c(t)$. Duhamel's formula and the fact that and $\sup_{(t,u)\in[0,t]\times E}\lVert\nabla h_n(u)\rVert\leq\sup_{(t,u)\in[0,t]\times E}\lVert\nabla h(u)\rVert$ is bounded imply that $\lVert V_n^tu-u\rVert\leq c(t)$ where $c(t)$ does not depend on $n$. We get that $\lVert\Pi_+^ku\rVert=\lVert \Pi^k_+V_n^tu-\Pi^k_+u\rVert\leq \lVert V_n^tu-u\rVert\leq c(t)$ and the result follows with $R(t)=r+c(t)$. 
	
	For the third claim we start by using the same argument as before to prove that $\{u\in C\,|\,\Pi_+^kV^tu=0\}$ is bounded. Now let $\{z_n\}\subset E$ be a sequence such that $$\Pi_kz_n\in X,\quad \Pi_-^kz_n=0,\quad\text{ and }\quad\Pi_+^kV^tz_n=0\quad\text{ for all }n\in \N.$$ We claim that $\{z_n\}$ has a convergent subsequence. First remark that, by the decomposition of $V_N$ in $E_N\oplus E^N$, for every $u \in E$ and $N\in \N$ we have $\Pi^NV_N^tu=\Pi^Nu$. Moreover, by definition of $z_n$, if $N\geq k$ then $\Pi_-^Nz_n=0$ and $\Pi_+^NV^tz_n=0$. For $N\geq k$ we have $$\lVert \Pi^Nz_n\rVert=\lVert \Pi^N_+z_n\rVert=\lVert\Pi_+^NV_N^tz_n\rVert=\lVert\Pi_+^NV_N^tz_n-\Pi_+^NV^tz_n\rVert\leq\lVert V_N^tz_n-V^tz_n\rVert.$$ Now $\{z_n\}_n$ is a bounded sequence so we can apply proposition \ref{prop: approximation} and for every $\epsilon>0$ there exists a $N_0(\epsilon)\in \N$ such that if $N\geq N_0$ then $\lVert V_N^tz_n-V^tz_n\rVert\leq \epsilon$. By the previous inequalities this implies that for $N\geq N_0$ we have $\lVert \Pi^Nz_n\rVert\leq\epsilon$. On the other hand, $\{z_n\}_n$ bounded implies that it has a weakly converging subsequence (still denoted by $\{z_n\}$ for simplicity) that converges when projected onto any finite dimensional subspace $E_N$. We conclude that for any $\delta>0$, with $\epsilon=\delta/3$ and $N\geq N_0(\epsilon)$, if $p,q\in \N$ are big enough we have $$\lVert z_p-z_q\rVert\leq \lVert \Pi_Nz_p-\Pi_Nz_q\rVert+\lVert \Pi^Nz_p\rVert+\lVert\Pi^Nz_q\rVert<\delta$$ which implies that $z_n$ is a Cauchy sequence.
	
\end{proof}

\begin{proof}[Proof of Proposition \ref{prop: infinite coisotropic}]
	Let $\mathcal V_\epsilon$ be the open $\epsilon$ neighbourhood of $\Pi_k(V^t(C)\cap\{\Pi_+^k=0\})$. We will show that for each $\epsilon>0$ there exists an $n\in\N$ such that $\Pi_k(V_n^t(C)\cap\{\Pi_+^k=0\})\subseteq \mathcal V_\epsilon$. Once this is proven, Lemma \ref{lemma: finite-infinite} part 1 and monotonicity of the capacity $\gamma$ imply that $c(X)\leq \gamma(\mathcal V_\epsilon)$ for every $\epsilon>0$ so $c(X)\leq\lim_{\epsilon\rightarrow 0}\gamma(\mathcal{V}_\epsilon)$. We then use that $\Pi_k(V^t(C)\cap\{\Pi_+^k=0\})$ is compact by Lemma \ref{lemma: finite-infinite} part 3 to conclude that $\lim_{\epsilon\rightarrow 0}\gamma(\mathcal{V}_\epsilon)=\gamma(\Pi_k(V^t(C)\cap\{\Pi_+^k=0\}))$ which is the desired result.
	
	The proof is by contradiction. Suppose that there exist an $\epsilon_0>0$ and a sequence $\{z_n\}\subset E$ such that for all $n\in \N$ $$\Pi_kz_n\in X,\quad \Pi_-^kz_n=0,\quad\Pi_+^kV_n^tz_n=0\quad\text{ and }\quad d(\Pi_kV_n^tz_n,\mathcal V_0)\geq \epsilon_0.$$ We claim that $\{z_n\}$ has a convergent subsequence. We use the same argument as in Lemma \ref{lemma: finite-infinite} part 3. For $N\geq k$ we have $$\lVert \Pi^Nz_n\rVert=\lVert \Pi^N_+z_n\rVert=\lVert\Pi_+^NV_N^tz_n\rVert=\lVert\Pi_+^NV_N^tz_n-\Pi_+^NV_n^tz_n\rVert\leq\lVert V_N^tz_n-V_n^tz_n\rVert.$$ By Lemma \ref{lemma: finite-infinite} part 2 we know that ${z_n}$ is a bounded sequence so we can apply Proposition \ref{prop: approximation} and for every $\delta>0$ there exists a $N_0(\delta)\in \N$ such that if $n,N\geq N_0$ then $\lVert V_N^tz_n-V_n^tz_n\rVert\leq \delta$. By the previous inequalities this implies that for $n,N\geq N_0$ we have $\lVert \Pi^Nz_n\rVert\leq\delta$. On the other hand, $\{z_n\}$ bounded implies that it has a weakly converging subsequence (still denoted by $\{z_n\}$ for simplicity) that converges when projected onto any finite dimensional subspace $E_N$. We conclude that for any $\delta>0$, with $\epsilon=\delta/3$ and $N\geq N_0(\epsilon)$, if $p,q\geq N_0$ are big enough we have $$\lVert z_p-z_q\rVert\leq \lVert \Pi_Nz_p-\Pi_Nz_q\rVert+\lVert \Pi^Nz_p\rVert+\lVert\Pi^Nz_q\rVert<\delta$$ which implies that $z_n$ is a Cauchy sequence. Denote $z$ its limit in $E$. The set $X$ is closed so $\Pi_kz\in X$ and $\Pi^k_-$ is continuous so $\Pi^k_-z=0$. This means that $z$ is an element of $C$. Moreover remark that $$\lVert V^tz-V^t_nz_n\rVert\leq \lVert V^tz-V^tz_n\rVert+\lVert V^tz_n-V^t_nz_n\rVert.$$  so by continuity of $V^t$ and again proposition \ref{prop: approximation} we get that $V^t_nz_n$ converges to $V^tz$ in $E$. Using the hypothesis $\Pi_+^kV_n^tz_n=0$ we find that $\Pi_+^kV^tz=0$ which allows us to conclude that $\Pi_kV^tz$ belongs to $\mathcal V_0$. This contradicts the fact that $d(\Pi_kV_n^tz_n,\mathcal V_0)\geq \epsilon_0>0$ for all $n\in \N$ achieving the proof of the theorem.
\end{proof}

\bigskip
\footnotesize
\noindent\textit{Acknowledgments.}
I am thankful to Emmanuel Opshtein for his useful comments concerning the proof of theorem \ref{thm:coisotropic nonsqueezing}. I also thank Claude Viterbo for his encouragement and for many helpful conversations. This work is part of my PhD funded by PSL Research University.

\normalsize
\appendix
\section{Some calculations of symplectic capacities}\label{appendix: calculations}
By definition we know that for every symplectic capacity $c$ we have $$c(B^{2n}_r)=\pi r^2=c(B_r^2\times \C^{n-1})$$
The reader interested in the proof of this equality for the two different symplectic capacities $c$ and $\gamma$ that were defined in \cite{viterbo1992symplectic} may look, for example, at \cite{aebischer1994symplectic}. We are interested in the value of Viterbo's capacities on coisotropic spaces $\C^k\times \R^{n-k}\subseteq \C^k\times \C^{n-k}$ with $n\neq k$. Recall that $c$ and $\gamma$ are first defined on open bounded sets $U$, then if $V$ is open and unbounded subsets then $c(V)$ is defined as the supremum of the values of $c(U)$ for all open bounded $U$ contained in $V$ and finally if $X$ is an arbitrary domain of $\C^n$ then $c(X)$ is the infimum of all the values $c(V)$ for all open $V$ containing $X$.

\begin{prop}\label{prop: coisotropic capacity} Consider the coisotropic subspace $\C^k\times \R^{n-k}\subseteq \C^k\times \C^{n-k}$ with $0\leq k<n$. We have $$c(\C^k\times \R^{n-k})=0=\gamma(\C^k\times \R^{n-k}).$$
\end{prop}
\begin{proof}
	First remark that for every $\lambda\neq 0$ we have $\lambda\cdot (\C^k\times \R^{n-k})=\C^k\times \R^{n-k}$ so by homogeneity of symplectic capacities we deduce that any capacity is either $0$ or $+\infty$ on coisotropic subspaces. Since we have the inequality $c(\C^k\times \R^{n-k})\leq \gamma(\C^k\times \R^{n-k})$ it is enough to prove that $\gamma(\C^k\times \R^{n-k})<+\infty$. By definition $$\gamma(\C^k\times \R^{n-k})=\inf\{\gamma(V)\,|\, V\quad \text{is open and} \quad \C^k\times \R^{n-k}\subseteq V\},$$ so it is enough to find an open set $V$ containing $\C^k\times \R^{n-k}$ with finite $\gamma$ value. Recall moreover that for a bounded open set we have $$\gamma(U)=\inf\{\gamma(\psi),\; \psi\in \ham^c(\C^n),\; \psi(U)\cap U=\emptyset\}$$ In order to find the open set with finite displacement energy we will use \cite[Proposition 4.14]{viterbo1992symplectic} which states the following: for a $\mathcal C^2$ compactly supported Hamiltonian $H:[0,1]\times \C^n\rightarrow \R$ that generates a flow $\psi_1$ we have $\gamma(\psi_1)\leq \lVert H\rVert_{\mathcal C^0}$. 
	
	Find a smooth function $f:\R\rightarrow \R$ with values on $]0,1[$ and $f'(s)>0$ for every $s\in \R$. Define the open set $$V=\{(q_1,p_1,\dots,q_n,p_n)\in \C^n\quad\text{such that}\quad |p_n|<f'(q_n)\}.$$ By hypothesis $k<n$ so $\C^k\times\R^{n-k}\subseteq V$. We claim that $\gamma(V)<+\infty$. For this we will consider the bounded Hamiltonian $H(q,p)=-2f(q_n)$ which generates the flow $$\psi_t(q,p)=(q_1,p_1,\dots,q_n,p_n+t2f'(q_n)).$$ If $(q,p)\in V$ then $$|p_n+2f'(q_n)|\geq 2f'(q_n)-|p_n|>f'(q_n)$$ which implies that $\psi_1(V)\cap V=\emptyset$. Let $U$ be an open bounded set contained in $V$, we have $\psi_1(U)\cap U=\emptyset$. Find a compactly supported smooth function $\chi:\C^n\rightarrow \R$ with values on $[0,1]$ and constant equal to $1$ on a neighbourhood of $\bigcup_{t\in [0,1]}\psi_t(U)$. Then $\chi H$ verifies $\lVert \chi H\rVert_{\mathcal C^{0}}\leq \lVert H\rVert_{\mathcal C^0}$ and by construction its flow still displaces the open set $U$. We conclude by \cite[Proposition 4.14]{viterbo1992symplectic} that $\gamma(U)\leq \lVert H\rVert_{\mathcal C^0}$. Since the bound does not depend on $U$ this implies that $\gamma(V)\leq \lVert H\rVert_{\mathcal C^0}$ which finally gives $$\gamma(\C^k\times \R^{n-k})\leq \lVert H\rVert_{\mathcal C^0}<+\infty$$ concluding the proof.
\end{proof}

\section{A Hamiltonian subgroup of the group of symplectic diffeomorphisms}\label{appendix: subgroup}
In this section we exhibit a subgroup of $Sympl(\C^n)$ which is strictly bigger than  the group of compactly supported Hamiltonian diffeomorphisms and whose elements are generated by sub-quadratic functions.

\begin{prop}Denote $Ham^{dL}(\C^n)$ the set of Hamiltonian diffeomorphisms $\varphi_t^H$ such that $H_t$, $\varphi_t^H$ and $(\varphi_t^H)^{-1}$ are all Lipschitz in space over compact time intervals. Then $Ham^{dL}(\C^n)$ is a subgroup of $Sympl(\C^n)$. Moreover $Ham^{dL}(\C^n)$ is strictly bigger than the group of compactly supported Hamiltonian diffeomorphisms.
\end{prop}
\begin{rem}
	The superscript $dL$ on $Ham^{dL}(\C^n)$ stands for double Lipschitz condition.
\end{rem}
\begin{proof}
	First recall the following formulas: $$\varphi_t^H\circ\varphi_t^K=\varphi_t^{H\#K}\quad\text{and}\quad (\varphi_t^H)^{-1}=\varphi_t^{\bar H},$$ where $$H\#K(t,z)=H(t,z)+K(t,(\varphi_t^H)^{-1}(z)),$$ $$\bar H(t,z)=-H(t,\varphi_t^H(z)).$$ The identity is clearly in $Ham^{dL}(\C^n)$ and it is an easy exercise to use these formulas to prove that $Ham^{dL}(\C^n)$ has a group structure. For the second statement, consider a Lipschitz autonomous Hamiltonian $H$ with Lipschitz gradient and use Gronwall's lemma to prove that $\varphi^H_t$ (and therefore $(\varphi_t^H)^{-1}=\varphi_{-t}^H$) is Lipschitz. 
\end{proof}

\end{document}